\documentclass[12pt]{amsart}
\usepackage{amssymb,amsmath}
\usepackage{geometry}    
\usepackage{mathrsfs}

\usepackage{graphicx}

\usepackage{vmargin}
\setmargrb{1in}{1in}{1in}{1in}

\newtheorem{theorem}{Theorem}[section]
\newtheorem{lemma}{Lemma}
\newtheorem{proposition}{Proposition}
\newtheorem{corollary}{Corollary}
\theoremstyle{definition}
\newtheorem{definition}{Definition}
\newtheorem{remark}{Remark}
\newtheorem{example}{Example}

\newtheorem{problem}{Problem}

\newcommand{\scp}[1]{\langle#1\rangle}

\begin{document}

	\title[Siegel's Lemma
	and best approximation points]{Applications of Siegel's Lemma to 
		a system of linear forms and its minimal points}
	
	\author{Johannes Schleischitz}

	\thanks{Middle East Technical University, Northern Cyprus Campus, Kalkanli, G\"uzelyurt \\
		johannes@metu.edu.tr ; jschleischitz@outlook.com}
	% EMAIL CORRETCED AFTER 2nd SUBMISSION   20.5.2020

	%\vspace{8mm}

\begin{abstract}
	Consider a real matrix $\Theta$ consisting of
	rows $(\theta_{i,1},\ldots,\theta_{i,n})$, for $1\leq i\leq m$. 
	The problem of making the system linear forms $x_{1}\theta_{i,1}+\cdots+x_{n}\theta_{i,n}-y_{i}$ 
	for integers $x_{j},y_{i}$ small
	naturally induces an ordinary and a uniform exponent
	of approximation, denoted by $w(\Theta)$ and $\widehat{w}(\Theta)$ respectively. 
	For $m=1$,
	a sharp lower bound for the ratio $w(\Theta)/\widehat{w}(\Theta)$
	was recently established by Marnat and Moshchevitin.
	We give a short, new proof of this result upon a hypothesis on
	the best approximation integer vectors associated to $\Theta$. 
	Our bound applies to general $m>1$, but is probably not
	optimal in this case.
	Thereby we also complement a similar conditional result of Moshchevitin, 
	who imposed a different assumption on the best approximations.
	Our hypothesis is satisfied in particular for $m=1, n=2$ and
	unconditionally confirms a previous observation of Jarn\'ik. 
	We formulate our results in a very general context of 
	approximation of subspaces of Euclidean spaces by lattices.
	We further establish criteria upon which a given number $\ell$ 
	of consecutive best approximation vectors are
	linearly independent.
	Our method is based on Siegel's Lemma.
\end{abstract}

\maketitle

{\footnotesize{

		Math Subject Classification 2010: 11J13, 11J82}}

%\vspace{1mm}

\section{A system of linear forms}

\subsection{Exponents of approximation and minimal points} \label{intro}

A standard problem in Diophantine approximation is, for
$mn$ given real numbers $\theta_{i,j}$, $1\leq i\leq m, 1\leq j\leq n$, 
to study simultaneously small absolute values of $m$ linear forms
\[
\theta_{i,1}x_{1}+\cdots+\theta_{i,n}x_{n} + y_{i}, \qquad\quad 1\leq i\leq m,
\]
with integers $x_{i}, y_{i}$ not all $0$.
Let $\Theta\in \mathbb{R}^{m\times n}$ be the corresponding matrix
and
$\underline{\theta}_{i}=(\theta_{i,1},\ldots,\theta_{i,n})$ for $1\leq i\leq m$
its rows.
Define the extended
matrix $\Theta^{E}\in \mathbb{R}^{m\times (m+n)}$ by gluing the $m\times m$
identity matrix to the right of $\Theta$, i.e. the $i$-th row 
$\underline{\theta}^{E}_{i}$ of $\Theta^{E}$ equals
\[
\underline{\theta}^{E}_{i}=(\underline{\theta}_{i}, \underline{e}_{i} )\in \mathbb{R}^{n+m},
\]
for $\underline{e}_{i}$ the $i$-th canonical base vector in $\mathbb{R}^{m}$.
Further writing $\underline{x}=(x_{1},\ldots, x_{n})$ 
and $\underline{y}=(y_{1},\ldots, y_{m})$ and  $\underline{z}=(\underline{x},\underline{y})$ which lies
in $\mathbb{Z}^{n+m}$,
our system can be equivalently written as
\begin{equation}  \label{eq:lifo}
\Theta^{E}\cdot \underline{z}.
\end{equation}
In this paper we consider
Euclidean spaces $\mathbb{R}^{N}$ equipped with the maximum norm
$\Vert \underline{\xi}\Vert= \max_{1\leq i\leq N} |\xi_{i}|$ 
for a vector $\underline{\xi}=(\xi_{1},\ldots,\xi_{N})$ in it, 
for simplicity we make no notational reference of the involved dimension $N$. 
However,
all results remain true when considering any other 
norm.   
All vectors below are considered column vectors for convenience.

First we recall the notion of classical exponents of approximation
associated to $\Theta$. According to \eqref{eq:lifo},
we 
denote by $w(\Theta)$ the supremum of $w$ such that 
\[
\Vert \Theta^{E}\underline{z}\Vert < \Vert \underline{z}\Vert^{-w}
\]
holds for certain integer vectors $\underline{z}$ of arbitrarily large norm
$\Vert \underline{z}\Vert$. 
We further define the uniform exponent
$\widehat{w}(\Theta)$ as the supremum of real 
parameters $\widehat{w}$ so that
the estimate 
\[
\Vert \underline{z}\Vert\leq X, \qquad \Vert\Theta^{E}\cdot \underline{z}\Vert < X^{-\widehat{w}}
\]
has a solution $\underline{z}\in\mathbb{Z}^{n+m}\setminus \{\underline{0}\}$ for all large $X$. 
%Again we may assume
%that $\underline{z}=\underline{z}(X)$ is some minimal point.
These exponents satisfy the relations
\begin{equation}  \label{eq:diri}
\infty\geq w(\Theta)\geq
\widehat{w}(\Theta)\geq \frac{n}{m}
\end{equation}
by a pigeon hole principle argument as in 
Dirichlet's Theorem (or Minkowski's Convex Body Theorem). More precisely, for every parameter $X>1$ there exists 
some $\underline{z}\in\mathbb{Z}^{n+m}\setminus \{\underline{0}\}$ 
of norm $\Vert \underline{z}\Vert\leq X$ for which the
vector \eqref{eq:lifo} has norm $\ll_{m,n} X^{-n/m}$. The notation $A\ll_{.} B$ always
means that there is a constant $c=c(.)$ depending only on the index variables
so that $A\leq cB$. Moreover $A\asymp_{.} B$
means $A\ll_{.} B\ll_{.} A$. 
We write $\ll,\gg, \asymp$ without index if the implied constants are absolute. 
We remark that $\widehat{w}(\Theta)=1$
if $m=n=1$ and $\Theta\notin \mathbb{Q}$,
hence $\widehat{w}(\Theta)\leq 1$ for $n=1$ and any $m$ and 
vectors $\Theta\notin \mathbb{Q}^{m}$,
whereas $\widehat{w}(\Theta)=\infty$ occurs for certain 
matrices with algebraically independent entries as soon as $n>1$.
As customary we
call $\Theta$ very well approximable if $w(\Theta)>n/m$, and
recall that $\Theta$ is singular in the sense of Diophantine approximation if 
$\widehat{w}(\Theta)>n/m$ (the
definition of singularity uses a slightly weaker condition though).

The exponents $w,\widehat{w}$ are closely related to best approximations
that we discuss now. 
Assume the columns of $\Theta^{E}$ are $\mathbb{Q}$-linearly independent,
that is $\Theta$ is non-degenerate in the sense of Jarn\'ik~\cite{jarnik2}.
Then $\Theta$ induces a sequence of points
$\underline{z}=(\underline{x},\underline{y})$
in $\mathbb{Z}^{n+m}$ of increasing norms, which we denote by $(\underline{z}_{k})_{k\geq 1}$, with the property that $\Vert \Theta^{E} \underline{z}_{k}\Vert>0$  
minimizes $\Vert \Theta^{E} \underline{z}\Vert$ upon all choices of $\underline{z}\in\mathbb{Z}^{n+m}\setminus \{\underline{0}\}$ with $\Vert \underline{z}\Vert<\Vert\underline{z}_{k+1}\Vert$
(note that this sequence depends on the chosen norm). 
Even for non-degenerate $\Theta$,
it may still happen that 
$\theta_{i,j}$ are $\mathbb{Q}$-linearly dependent together
with $\{1\}$ and then the sequence $(\underline{z}_{k})_{k\geq 1}$
may not be unique (up to sign).
For simplicity we also want to exclude this case,
however remark that some of our results might extend
to non-degenerate matrices upon choosing any appropriate sequence
in case of ambiguity. If the sequence is well-defined,
we say $\Theta$ is ''good'', the terminology originates in~\cite{ngm}.
So in the sequel we always assume $\Theta$ is good,
i.e. induces a uniquely determined  (up to sign)
sequence $(\underline{z}_{k})_{k\geq 1}$,
and shall call these {\em best approximations} or {\em minimal points} 
associated to $\Theta$. 
The sequence of minimal points obviously satisfies
\begin{equation}  \label{eq:ndef}
\Vert \underline{z}_{1}\Vert< \Vert\underline{z}_{2}\Vert< \cdots, \qquad\qquad \Vert \Theta^{E}\underline{z}_{1}\Vert>\Vert \Theta^{E}\underline{z}_{2}\Vert>\cdots.
\end{equation}
It is easy to see
that we can choose the vectors $\underline{z}$ realizing 
the exponents $w(\Theta), \widehat{w}(\Theta)$
among the sequence $(\underline{z}_{k})_{k\geq 1}$. 
Regarding the latter, more precisely
for any $\epsilon>0$ we have
\begin{equation}   \label{eq:didit}
0<\Vert\Theta^{E}\underline{z}_{k}\Vert \leq \Vert \underline{z}_{k+1}\Vert^{-\widehat{w}(\Theta)+\epsilon } <
\Vert \underline{z}_{k}\Vert^{-\widehat{w}(\Theta)+\epsilon}, \qquad \qquad k\geq k_{0}(\epsilon),
\end{equation}
complementary to the well-known estimates
\[
0<\Vert\Theta^{E}\underline{z}_{k}\Vert \ll_{m,n} \Vert \underline{z}_{k+1}\Vert^{-n/m} <
\Vert \underline{z}_{k}\Vert^{-n/m}, \qquad \qquad k\geq 1,
\]
that are slightly stronger than \eqref{eq:didit} if $\widehat{w}(\Theta)=n/m$.

The sequence of minimal points has already been investigated 
by Jarn\'ik, see for example~\cite{jarnik1}.
Important special case are $m=1$ and $n=1$.
For one linear form, i.e. $m=1$, this sequence was studied 
by Davenport and Schmidt~\cite{davsh},~\cite{ds2} when studying approximation to a real number by algebraic integers. In the same paper they
also dealt with the
analogous sequence with respect to the dual setting
of simultaneous approximation, corresponding to $n=1$ in our notation.
Investigation of the latter was 
emphasized with contributions
by several authors, including a series
of papers by Lagarias starting from~\cite{lag0} in 1979 and later Moshchevitin.
We also refer to the more recent paper by Chevallier~\cite{chev} 
for an introduction to the simultaneous approximation setting, including a wealth of references. The general case of arbitrary $m,n$ 
has been studied 
for example in~\cite{moshus},~\cite{ngm} and several other 
papers by Moshchevitin that we
will recall later.

\subsection{Outline of the paper}

One purpose of this paper is to study lower bounds for the 
quotient $w(\Theta)/\widehat{w}(\Theta)$. This topic already goes back to 
Jarn\'ik~\cite{jarnik1},\cite{jarnik2}
and has attracted interest lately, it gave rise to a series
of papers within the last decade including~\cite{mo2012},\cite{ss},\cite{ssch},\cite{germos},\cite{mamo},\cite{nr}, particularly on the cases $m=1$ or $n=1$. 
We survey known results in Section~\ref{s2}, in particular Marnat and Moshchevitin~\cite{mamo} and Moshchevitin~\cite{moneu}. Thereby we encompass
results on linear independence of minimal points.
In Section~\ref{s3}, we will establish a new complementary 
conditional bound and compare it with~\cite{mamo},~\cite{moneu}.
In certain cases our result is unconditional and  
implies observations of
Jarn\'ik~\cite{jarnik1}, that in turn is 
the special case $n=2$ of Theorem~\ref{anni} below 
that originates in~\cite{mamo}.

The second purpose is to use a very similar method to
derive conditions under which 
a given number $\ell$ of {\em consecutive} minimal points are linearly
independent. For $\ell=n+m$, this question results in studying the regularity
of the quadratic matrices whose columns are these best approximations, 
a classical topic. In Section~\ref{cri},
we want to treat the case $\ell<n+m$, with emphasis on $\ell=3$. 

\section{Known lower bounds for $w(\Theta)/\widehat{w}(\Theta)$} \label{s2}

\subsection{Cases $m=1$ or $n=1$}

In these cases $m=1$ or $n=1$,
the minimum ratio $w(\Theta)/\widehat{w}(\Theta)$
was established by Marnat and Moshchevitin~\cite{mamo}
(see also Rivard-Cooke's PhD-thesis~\cite{rk} for
a different proof). We only state their result for $m=1$.

\begin{theorem}[Marnat, Moshchevitin]  \label{anni}
	Let $n\geq 2$. If $m=1$, for any good 
	$\Theta\in \mathbb{R}^{m\times n}$ we have
	\[
	\frac{ w(\Theta) }{ \widehat{w}(\Theta) } \geq G_{1,n}
	\]
	where $G_{1,n}= G_{1,n}(\widehat{\omega}(\Theta))$ is the unique positive real root of $P_{1,n}(x)=1- \widehat{w}(\Theta) +\sum_{j=1}^{n-1} x^j$.
	Equality is attained for certain $\Theta$, thus the bound is optimal.
\end{theorem}

\subsection{Linear independence of minimal points}  \label{linin}
In this section, we prepare some notation
and further survey some more facts on minimal points.
The topic of linear independence of subsets of best approximations $\underline{z}_{j}$ is related to the exponents $w(\Theta), \widehat{w}(\Theta)$ and
has been investigated in Diophantine approximation. It is easy to see that any
two consecutive best approximations $\underline{z}_{k},\underline{z}_{k+1}$
are linearly independent, a short argument in fact shows that 
any such pair spans (as a $\mathbb{Z}$-module) the lattice obtained from intersecting their real
span real with $\mathbb{Z}^{n+m}$, see~\cite[Lemma~4]{chev} 
(there the case of simultaneous approximation $n=1$ is treated, but for 
any system of linear
forms an analogous argument applies). 
On the other hand, it may happen that all large best approximations
lie in a fixed space of dimension only $2$, see Theorem~\ref{nmo} below. 
The next definition deals in more detail with the dimension of the sublattice of $\mathbb{Z}^{n+m}$ spanned by minimal points.
Recall the notion of a good matrix from Section~\ref{intro}.

\begin{definition}  \label{def1}
	Let $n\geq 1, m\geq 1$ integers and consider good
	matrices $\Theta\in \mathbb{R}^{m\times n}$
	as above.
	Let 
	\[
	R(\Theta) = \min\{ h: \exists k_{0}\;\; \text{such that the vectors}\;\; (\underline{z}_{k})_{k\geq k_{0}} \;\;
	\text{span a space of dimension}\;\; h \}.
	\]
	We denote this real subspace by $\mathscr{S}_{\Theta}\subseteq \mathbb{R}^{n+m}$.	For $1\leq h\leq m+n$, define 
	\[
	\mathscr{G}_{h}=\mathscr{G}_{h}^{m,n}=
	\{\Theta: R(\Theta)=  h\}, \qquad
	\mathscr{H}_{h}=\mathscr{H}_{h}^{m,n}=
	\{\Theta: R(\Theta)\leq  h\}.
	\]

\end{definition}

The notation $R(\Theta)$ was introduced in~\cite{ngm}.
Obviously $\mathscr{G}_{h}$ are disjoint in $h$ and
\[
\emptyset=\mathscr{H}_{1}\subseteq \mathscr{H}_{2}\subseteq\cdots \subseteq \mathscr{H}_{n+m}=\mathbb{R}^{m\times n}, \qquad\qquad m,n\geq 1.
\]
%
%and from the quoted result in Section~\ref{intro} we have
%$\mathscr{H}_{1}^{m,n}=\emptyset$ for every $m,n\geq 1$ and %$\mathscr{H}_{2}^{1,n}= \emptyset$ for $n\geq 2$, whereas
%$\mathscr{G}_{3}^{1,n}=\mathscr{H}_{3}^{1,n}\neq \emptyset$ for $n\geq 3$ by %Theorem~\ref{mosh}.

Clearly, the generic case is $\Theta\in \mathscr{G}_{m+n}$.
As observed above, $R(\Theta)\geq 2$ for any $\Theta$
and equality occurs in some cases.
Combining claims from~\cite[Section~2.2]{ngm}
we get the following almost complete classification on $R(\Theta)=2$.

\begin{theorem}[Moshchevitin]  \label{nmo}
	If $2\leq n<m$, the set of good matrices in $\mathscr{G}_{2}$ as a subset of $\mathbb{R}^{mn}$ has Hausdorff dimension at least $m(n-1)$, in particular is not empty. If $n>m$ then $\mathscr{G}_{2}=\emptyset$. If the entries of $\Theta$ are algebraically independent, then $\Theta\notin \mathscr{G}_{2}$.
\end{theorem}

See also~\cite[Section~1.3]{mosh}
for a sketch of the proof when $m=1$. The same is true if the elements of $\Theta$ are algebraically independent~\cite[Corollary~2]{ngm}.
There are also partial results on the open case $m=n$ in~\cite{ngm}.
The next theorem comprises two more results by Moshchevitin~\cite[Theorem~14]{ngm}, the first deals with the case of best approximations ultimately
lying in a $3$-dimensional
sublattice of $\mathbb{Z}^{n+m}$
and is similarly surprising.

\begin{theorem}[Moshchevitin]  \label{mosh}
	For every $m\geq 1, n\geq 3$, there exist uncountably many good $\Theta\in\mathscr{G}_{3}$.
	On the other hand, for $n=1$ and $m\geq 1$, the set
	$\mathscr{H}_{m+n-1}=\mathscr{H}_{m}$ is empty.  
\end{theorem}

Even in the case $m=1$ also available in~\cite{mosh1,mosh}, where we 
can obviously take $n=2$ as well, 
the first result is very surprising.
Reverse to the latter claim of the theorem, for $n=1$
the corresponding determinants formed by $m+1$ consecutive
minimal points may be $0$ for $k\geq k_{0}$, 
see again Moshchevitin~\cite{mosh}. 

\subsection{A conditional bound by Moshchevitin}  \label{aco}
For $\min\{ m,n\}>1$, the optimal bound for the ratio $w(\Theta)/\widehat{w}(\Theta)$ is unknown. Some 
unconditional estimates, as well as counterexamples for 
reasonably sharper estimates, due to Jarn\'ik %~\cite{jarnik2} 
are recalled 
in~\cite[Section~3.1]{moneu}. 
Here we just explicitly want to state
\begin{equation} \label{eq:just}
w(\Theta) \geq \widehat{w}(\Theta)^{n/(n-1)} - 3\widehat{w}(\Theta), \qquad 
\text{if} \; n\geq 3\; \text{and} \; \widehat{w}(\Theta)\geq (5n^2)^{n-1}.
\end{equation}
% CORRECTED AFTER SUBMISSION  oct 18, 2021  !!!!!!!!!!!!!!!
See also Moshchevitin~\cite{mo2012}
when $m\geq 3, n=2$ and $\widehat{w}(\Theta)\geq 1$.

Recall $\mathscr{G}_{h}$ and $\mathscr{S}_{\Theta}$ from Definition~\ref{def1}.
To formulate a result
indicated in~\cite{moneu}, we consider the following
subsets of matrices $\Theta\in\mathscr{G}_{h}$ obtained from
some (rather strong) linear independence property.

\begin{definition} \label{def2}
	Let $1\leq h\leq m+n$ be an integer.
	Let $\textbf{L}_{h}\subseteq \mathscr{G}_{h}$ be the set of matrices
	within $\mathscr{G}_{h}$
	with the property that
	for infinitely many integers $t\geq 1$,
	the $h$ consecutive best approximations
	$\underline{z}_{t},\underline{z}_{t+1},\ldots,  \underline{z}_{t+h-1}$ are linearly independent (thus span $\mathscr{S}_{\Theta}$). 
	If the property holds for all large $t$, we call the corresponding 
	smaller set
	$\textbf{SL}_{h}\subseteq \mathscr{G}_{h}$.
\end{definition}

We will omit the dependence on $m,n$ in the notation.
We want to point out that 
results quoted in Section~\ref{linin} imply for small $h$ the identities
\[
\mathscr{G}_{2}= \textbf{L}_2=\textbf{SL}_{2}, \qquad \mathscr{G}_{3}= \textbf{L}_3.
\]
However $\textbf{SL}_3\subsetneq \mathscr{G}_{3}$.
Under the
assumption $\Theta\in\textbf{L}_h$,
as pointed out to the author by the referee, there is a short argument
giving a strong bound for the ratio $w(\Theta)/\widehat{w}(\Theta)$.
For $a,b$ positive integers and $\Theta\in \mathbb{R}^{m\times n}$, 
denote by $G_{a,b}$ the positive root of
\[
P_{a,b}(x)=-\sum_{j=1}^{a-1} \frac{ \widehat{w}(\Theta)}{x^j} +1- \widehat{w}(\Theta) +\sum_{j=1}^{b-1} x^j = 0.  
\] 
When $m=1$,
the definition agrees with $G_{1,n}$ defined in Theorem~\ref{anni}.
For $\Theta\in\textbf{L}_{h}$ recall the $h$-dimensional subspace
$\mathscr{S}_{\Theta}\subseteq \mathbb{R}^{m+n}$ 
from Definition~\ref{def1} and
associate to $\Theta$ another subspace in $\mathbb{R}^{m+n}$ given as
\[
\mathscr{L}_{\Theta}= \{ (\underline{z},\Theta \underline{z}): \underline{z}\in \mathbb{R}^{n} \}
\subseteq \mathbb{R}^{m+n},
\]
and derive
\[
n^{\prime}= \dim( \mathscr{S}_{\Theta} \cap \mathscr{L}_{\Theta}), \qquad m^{\prime}= h-n^{\prime}.
\]
Notice that in the generic situation $h=m+n$ we just have $\mathscr{S}_{\Theta}= \mathbb{R}^{m+n}$,
and $m=m^{\prime}, n=n^{\prime}$.
With this notation, the following bounds for $\Theta\in\textbf{L}_{h}$ hold.

\begin{theorem}[Moshchevitin (essentially)]    \label{mothm}
	Let $1\leq h\leq m+n$. For $\Theta\in\textbf{L}_{h}$
	the estimate
	\begin{equation}  \label{eq:neujaa}
	\frac{w(\Theta)}{\widehat{w}(\Theta) } \geq G_{m^{\prime},n^{\prime} }
	\end{equation}
	holds.
	In particular, in the generic case $h=m+n$, we have
	\begin{equation} \label{eq:neujaa0}
	\frac{w(\Theta)}{\widehat{w}(\Theta) } \geq G_{m,n}.
	\end{equation}
\end{theorem}

As pointed out,
claim \eqref{eq:neujaa0} is already stated, but without proof, in~\cite[Section~3.5]{moneu}. 
If $m=1$ this simplifies to the sharp (unconditional) bound
in Theorem~\ref{anni}.
The bound is likely to be optimal, possibly even without
assumption $\Theta\in\textbf{L}_{h}$, for general $m,n$. 
We provide a proof of Theorem~\ref{mothm} 
reported to the author by the referee in the Appendix in Section~\ref{annex}.

%\section{New uncondiitonal lower bounds for $w(\Theta)/\widehat{w}(\Theta)$ }
%
%First we give an unconditional bound that complements Jarn\'ik [cite] %estimates
%
%\begin{theorem}  \label{TT}
%	For any $\Theta$ we have 
%	\begin{equation}  \label{eq:diessgute}
%	\widehat{w}(\Theta) \leq \frac{w(\Theta)}{\widehat{w}(\Theta)}+ %\left(\frac{w(\Theta)}{\widehat{w}(\Theta)}\right)^2 +\cdots+
%	\left(\frac{w(\Theta)}{\widehat{w}(\Theta)}\right)^{m+n-2}.
%	\end{equation}
%
%\end{theorem}

% PROOF MISSING

\section{New conditional lower bounds for $w(\Theta)/\widehat{w}(\Theta)$} \label{s3}

\subsection{General $h$ }  \label{r>3}

From $(\underline{z}_{k})_{k\geq 1}$ the sequence of best approximation associated
to $\Theta\in \mathbb{R}^{m\times n}$, derive
\begin{equation}  \label{eq:tauuuuu}
\overline{\tau}= \limsup_{k\to\infty} \frac{\log \Vert \underline{z}_{k+1}\Vert }{\log \Vert \underline{z}_{k}\Vert } \geq 1, \qquad \underline{\tau}= \liminf_{k\to\infty} \frac{\log \Vert \underline{z}_{k+1}\Vert }{\log \Vert \underline{z}_{k}\Vert } \geq 1.
\end{equation}
We introduce simplifying (strong) short vector hypotheses on $\Theta$ that
are slightly weaker than $\textbf{L}_h$ resp. $\textbf{SL}_h$,
for given $h$.
Assume $m,n$ are fixed in the sequel.

\begin{definition}  \label{zdef}
	Let $1\leq h\leq m+n$ be an integer.
	Let $\textbf{V}_h=\textbf{V}_h^{m,n}$
	be the set of good matrices $\Theta\in\mathbb{R}^{m\times n}$ 
	with the property that for infinitely many
	integers $t\geq 1$
	the lattice 
	\[
	\scp{ \underline{z}_t, \ldots, \underline{z}_{t+h-1} }_{\mathbb{R}} \cap \mathbb{Z}^{m+n} 
	\]
	obtained by intersecting the real span
	of the $h$ consecutive best approximations
	with the integer lattice, contains a short integer vector $\underline{v}=\underline{v}_t\in\mathbb{Z}^{m+n}$ of 
	norm $\Vert \underline{v}_t\Vert\ll \Vert \underline{z}_{t}\Vert^{o(1)}$ as $t\to\infty$. 
	If we assume the property for all large $t$, we denote 
	the induced smaller set by
	$ \textbf{SV}_h=\textbf{SV}_h^{m,n}$.
\end{definition}

We stress that we do not need to restrict $\Theta$ to $\mathscr{H}_h$ here.
For readability
we will again omit upper case indices $m,n$ in $\textbf{V}_h, \textbf{SV}_h$ below.
The conditions $\Theta\in \textbf{V}_h$ or $\Theta\in \textbf{SV}_h$ become less stringent the larger $h$ is. A generic matrix $\Theta\in\mathbb{R}^{m\times n}$ lies in $\textbf{SV}_{m+n}$, but in no
$\textbf{SV}_{b}$ with $b<m+n$.
It is not hard to see that
our conditions are indeed, at least formally, less restrictive than those in Definition~\ref{def2}. % improve!!!!!!!!!!

\begin{proposition}  \label{pro1}
	For every $1\leq h\leq m+n$, we have
	\[
	\textbf{L}_{h} \subseteq \textbf{V}_h,\qquad\quad \textbf{SL}_{h} \subseteq \textbf{SV}_h.
	\]
\end{proposition}

\begin{proof}
	We check the first inclusion only. Let $\Theta\in \textbf{L}_h$.
	For $k_{0}$ as in Definition~\ref{def1} and any $t$ that satisfies the hypothesis in Definition~\ref{def2},
	take the constant vector $\underline{v}_t= \underline{z}_{k_{0}}$. 
	Since by assumption $\Theta\in\mathscr{G}_{h}$, this vector 
	$\underline{v}_t$ lies 
	in the space $\mathscr{S}_{\Theta}$ spanned by $\underline{z}_{t},\underline{z}_{t+1},\ldots,  \underline{z}_{t+h-1}$,
	moreover
	it has absolutely bounded norm $\Vert \underline{v}_t\Vert=\Vert \underline{z}_{k_{0}}\Vert=O(1)=\Vert \underline{z}_t\Vert^{o(1)}$. Thus indeed $\Theta\in \textbf{V}_h$.
\end{proof}

We believe that the difference sets $\textbf{V}_h\setminus  \textbf{L}_h$
and  $\textbf{SV}_h\setminus  \textbf{SL}_h$
are non-empty, however we do not have examples at hand.
Our main result of this section below admits a short proof with Siegel's Lemma. We interpret $\infty/\infty= \infty$.

\begin{theorem} \label{better01}
	Let $h\geq 2$ and $m\geq 1,n\geq 1$ integers.
	Assume $\Theta\in \mathbb{R}^{m\times n}$ lies in $\textbf{V}_{h}$. Then
	\begin{equation}  \label{eq:diegute}
	\widehat{w}(\Theta) \leq 1+\frac{w(\Theta)}{\widehat{w}(\Theta)}+ \left(\frac{w(\Theta)}{\widehat{w}(\Theta)}\right)^2 +\cdots+
	\left(\frac{w(\Theta)}{\widehat{w}(\Theta)}\right)^{h-2}.
	\end{equation}
	More precisely
	we have
	\begin{equation}  \label{eq:nettz}
	\widehat{w}(\Theta) \leq 1 + \overline{\tau} + \overline{\tau}^{2}+ \cdots+ 
	\overline{\tau}^{h-2}= \frac{\overline{\tau}^{h-1} - 1}{\overline{\tau}-1}.
	\end{equation}
	If $\Theta\in \textbf{SV}_h$, then
	\begin{equation}  \label{eq:hallg}
	\widehat{w}(\Theta) \leq 1 + \underline{\tau} (1+\overline{\tau}+ \overline{\tau}^{2}+ \cdots+ 
	\overline{\tau}^{h-3})= 1+\underline{\tau}\frac{\overline{\tau}^{h-2} - 1}{\overline{\tau}-1}.
	\end{equation}
	In particular, if  $\Theta\in \textbf{SV}_h$
	and there is equality in \eqref{eq:diegute}, then 
	\begin{equation}  \label{eq:begg}
	\lim_{k\to\infty} \frac{\log \Vert \underline{z}_{k+1}\Vert }{\log \Vert \underline{z}_{k}\Vert }=
	\frac{w(\Theta)}{\widehat{w}(\Theta)}.
	\end{equation}
\end{theorem}

\begin{remark}
	When we relax the condition in Definition~\ref{zdef} to $\Vert \underline{v}_t\Vert\ll \Vert \underline{z}_{t}\Vert^{\Delta+o(1)}$ for 
	$\Delta\geq 0$ a parameter, the argument of the proof with minor adaptions
	implies the bound 
	\begin{align*}
	\widehat{w}(\Theta) &\leq 1+\frac{w(\Theta)}{\widehat{w}(\Theta)}+ \left(\frac{w(\Theta)}{\widehat{w}(\Theta)}\right)^2 +\cdots+
	\left(\frac{w(\Theta)}{\widehat{w}(\Theta)}\right)^{h-2}+ \frac{2\Delta w(\Theta)}{\underline{\tau}} \\
	&\leq 1+\frac{w(\Theta)}{\widehat{w}(\Theta)}+ \left(\frac{w(\Theta)}{\widehat{w}(\Theta)}\right)^2 +\cdots+
	\left(\frac{w(\Theta)}{\widehat{w}(\Theta)}\right)^{h-2}+  2\Delta w(\Theta) 
	\end{align*}
	in terms of $\Delta$, essentially since then $\varepsilon$
	from the proof can be replaced 
	by $\Delta w(\Theta)$.
	For $\Delta=0$ it naturally coincides with \eqref{eq:diegute}, for $\Delta\geq 1/2$ it becomes always trivial. 
\end{remark}

%Notice that the bound is just as in Theorem~\ref{TT} but increased 
%by one, if $h=m+n$. % Vllt weg
%
%
%
%
We see from $\overline{\tau}\leq w(\Theta)/\widehat{w}(\Theta)$, the claim of Lemma~\ref{lemur} below, that \eqref{eq:nettz} 
directly implies \eqref{eq:diegute}. 
We compare the claim with Theorem~\ref{mothm}.
For one linear form  $m=1$ and $h=m+n=n+1$,
we recognize \eqref{eq:diegute} again as equivalent
to \eqref{eq:neujaa0}, thus it again yields
the sharp bound in Theorem~\ref{anni}.
If $m>1$, our estimate is weaker than \eqref{eq:neujaa0}, however upon a more
moderate assumption.
We should remark that
refinements in the spirit of
\eqref{eq:nettz}, \eqref{eq:hallg} can be obtained
in Theorem~\ref{mothm} as well.
At this point we also want to refer again to Moshchevitin~\cite{mo2012}
for a stronger bound than \eqref{eq:diegute}
when $m\geq 3, n=2$ and $\widehat{w}(\Theta)\geq 1$.

A natural problem on the gap between Definition~\ref{zdef}
and Definition~\ref{def2} arises.

\begin{problem}
	Do we have the stronger estimate 
	\eqref{eq:neujaa} for any $\Theta\in \textbf{V}_h$?
\end{problem}

A consequence of Theorem~\ref{better} is that $\textbf{V}_h$ is small
if $n/m$ is large compared to $h$.

\begin{corollary} \label{reeh}
	Assume $m,n$ satisfy
	\begin{equation} \label{eq:n3m}
	n>(h-1)m.
	\end{equation}
	Then any $\Theta\in \textbf{V}_{h}$ is
	very well approximable, i.e. $w(\Theta)>n/m$. Hence, upon 
	\eqref{eq:n3m}, the set $\textbf{V}_{h}\subseteq \mathbb{R}^{mn}$
	has $mn$-dimensional
	Hausdorff measure $0$ (in fact Hausdorff dimension smaller than $mn$), and contains no matrix with only algebraic entries.
\end{corollary}

\begin{proof}
	We readily check that \eqref{eq:diri}, \eqref{eq:diegute} and \eqref{eq:n3m} implies
	$w(\Theta)>n/m$. The metric implication
	is then a well-known generalization of a result of Jarn\'ik~\cite{jarnik},
	see Beresnevich and Velani~\cite{bervel} for reasonably stronger versions.
	The claim for $\mathbb{Q}$-linearly independent algebraic matrices follows
	as they satisfy $w(\Theta)=n/m$ 
	by a direct consequence of Schmidt's Subspace Theorem (see~\cite[Theorem~2.8,~2.9]{bugbuch}).
\end{proof}

%If $h<n+1$,
%then our bound becomes stronger anyway. 
%If we assume the stronger property $\Theta\in \textbf{L}_{h}$ then
%application of Theorem~\ref{mothm} implies a stronger claim.

We wonder if we can relax the assumption
to $h<m+n$. We also include a speculation on the 
uniform exponent motivated by the construction in~\cite{mosh}.

\begin{problem}  \label{p1}
	Let $m\geq 1,n\geq 1$ and $h<m+n$. 
	Are all $\Theta\in \textbf{V}_{h}$ (if any exist) very well approximable, i.e. $w(\Theta)>n/m$?
	Does the set $\textbf{V}_{h}$ have $mn$-dimensional Lebesgue-measure $0$
	(Hausdorff dimension smaller than $mn$) and not contain algebraic matrices?
	Is the stronger conclusion $\widehat{w}(\Theta)>n/m$ true for any
	$\Theta\in \textbf{V}_{h}$ (at least for large $n/m$)?
\end{problem}

For $m=1$ and $\Theta\in \mathscr{H}_{h}$ in place of $\Theta\in\textbf{V}_{h}$, 
a positive answer concerning the ordinary exponent 
can be inferred from Theorem~\ref{anni}, with a similar deduction as Corollary~\ref{reeh} from Theorem~\ref{better01} (see also the appendix), as pointed out to the author in private
correspondence by N. Moshchevitin. Moreover, for any $m,n$ it is true
for $\Theta\in \textbf{L}_{h}$ as well by \eqref{eq:neujaa}, see
also the last paragraph of~\cite[Section~8]{akhmosh}.
On the other hand, for $m>1$ and general
matrices in $\mathscr{H}_{h}$ the problem seems open, as for
$\Theta\in \textbf{V}_h$ in Problem~\ref{p1}.

We finally remark that
similar, unconditional, quantitative claims as \eqref{eq:nettz}, \eqref{eq:hallg}, \eqref{eq:begg},
relating best approximations with classical exponents, were 
recently established for simultaneous approximation (i.e. $n=1$) 
by Nguyen, Poels and Roy~\cite{nr}.

\subsection{Special case $\Theta\in \textbf{V}_{2}$}

From Theorem~\ref{better01} we get

\begin{theorem} \label{slstr01}
	Let $\Theta\in \textbf{V}_{2}$.
	Then 
	\begin{equation} \label{eq:germol}
	\widehat{w}(\Theta) \leq 1.
	\end{equation}
	By \eqref{eq:ndef} in particular $n\leq m$.
\end{theorem}  

If we restrict to $\Theta\in \mathscr{G}_{2}= \textbf{L}_{2}$, we can
obtain a slightly stronger estimate already observed by  Moshchevitin~\cite[Theorem~8]{ngm}, with a new proof. If $(\underline{z}_{k})_{k\geq 1}$
is the sequence of best approximations 
associated to $\Theta\in \mathscr{G}_{2}$, we have
\begin{equation} \label{eq:germol2}
\Vert \Theta^{E}\underline{z}_{t}\Vert > c\cdot \Vert \underline{z}_{t+1}\Vert^{-1}, \qquad t\geq 1,
\end{equation}
for some $c=c(\Theta)>0$. That is a partial claim of Theorem~\ref{nmo} above. 
The stronger version follows from our proof of Theorem~\ref{better01} below, upon using the stronger assumption $\Vert \underline{v}_{t}\Vert= O(1)$ compared to Definition~\ref{zdef} valid for $\Theta\in \mathscr{G}_{2}=\textbf{L}_{2}$,
see the proof of Proposition~\ref{pro1}.

%The result applies in the more general setting analogous to
%Theorem~\ref{slstr00} as well, we do not formulate it.

\subsection{Special case $\Theta\in \textbf{V}_{3}$}  \label{r3}

We
derive a new proof a result of Jarn\'ik.

\begin{theorem} \label{better}
	Let $m\geq 1,n\geq 1$ integers and assume $\Theta\in \textbf{V}_{3}$. Then
	\begin{equation}  \label{eq:mybd}
	w(\Theta)\geq \widehat{w}(\Theta)^{2}-\widehat{w}(\Theta).
	\end{equation}
	In fact we have
	\begin{equation}  \label{eq:mybod}
	\widehat{w}(\Theta)\leq \overline{\tau}+1.
	\end{equation}
	If $\Theta\in \textbf{SV}_3$, then 
	\begin{equation}  \label{eq:mybod01}
	\widehat{w}(\Theta)\leq \underline{\tau}+1,
	\end{equation}
	in particular then equality in \eqref{eq:mybd} implies 
	\[
	\lim_{k\to\infty} \frac{\log \Vert \underline{z}_{k+1}\Vert }
	{\log \Vert \underline{z}_{k}\Vert }  =
	\frac{\widehat{w}(\Theta) }{w(\Theta)}.
	\]
\end{theorem}

If $m=1, n=2$, formula \eqref{eq:mybd} is unconditional and already occurs
in Jarn\'ik~\cite[Theorem~2]{jarnik1} with a different proof.
Jarn\'ik's result is the special case $n=2$ 
in the linear form result of Theorem~\ref{anni}.
As thankfully pointed out to the author by N. Moshchevitin, 
Jarn\'ik's proof can be extended to the more general
situation $\Theta\in \mathscr{G}_{3}=\textbf{L}_{3}$ for any $m,n$, which
however for general $m,n$ is still slightly weaker than
Theorem~\ref{better} where we assume $\Theta\in \textbf{V}_{3}$, at least formally.

%We remark that \eqref{eq:mybod} implies \eqref{eq:mybd} by 
%$\overline{\tau}\leq w(\Theta)/\widehat{w}(\Theta)$, see Lemma~\ref{lemur} below.
%Again the claims are applicable in the more general analogous setting of
%Theorem~\ref{slstr00}, we do not formulate it.

%The proof in fact shows that if for a matrix $\Theta\in\mathscr{H}_{3}^{m,n}$
%three associated consecutive minimal points $\underline{z}_{t}, %\underline{z}_{t+1}, \underline{z}_{t+2}$ span $\mathscr{S}_{\Theta}$, 
%then $\underline{z}_{t+2}$ has significantly larger
%norm %  (and smaller scalar product with $\underline{\theta}$) 
%than its preceding minimal point $\underline{z}_{t+1}$.

% letzten satz vll streichen !!!!!!!!!!!!!!!!!!!!!!!!

We briefly discuss the consequence of Corollary~\ref{reeh} 
for $h=3$. 
If $m=1$, 
any vector $\Theta\in \textbf{V}_{3}$ for $n\geq 3$
is very well approximable. For $\Theta\in \mathscr{G}_{3}= \textbf{L}_{3}\subseteq \textbf{V}_{3}$
this follows from Jarn\'ik~\cite{jarnik1} already.
On the other hand, any
matrix in $\textbf{V}_{2}$ induces
the upper bound $\widehat{w}(\Theta)\leq 1$ independent of $m,n$, see Theorem~\ref{slstr01} above. 
This clearly does not exclude that a matrix
in $\textbf{V}_{2}$ is very vell approximable.

For sake of completeness, we state a related result
on simultaneous approximation $n=1$ where our hypothesis
\eqref{eq:n3m} fails. Lagarias~\cite[Theorem~5.2]{lagarias}
showed that for $m=2$ and a badly approximable vector $\Theta=(\theta_{1},\theta_{2})^{t}\in\mathbb{R}^{2}$ there is an absolute upper bound
on the number of consecutive triples of linearly 
dependent minimal points $\underline{z}_{k}, \underline{z}_{k+1}, \underline{z}_{k+2}$. In the same paper he shows that the claim
is not true if the restriction to badly approximable vectors is dropped, see
also~\cite{mosh2} for a generalization.

\section{Criteria for linear independence of consecutive minimal points} \label{cri}

%\subsection{The $\mathbb{Q}$-linearly independent case} \label{31}

Let $\ell\geq 3$ be a given integer.
We study under which assumptions on a good matrix
$\Theta\in\mathbb{R}^{m\times n}$ 
we can deduce that $\ell$ consecutive
minimal points $\underline{z}_{k}, \ldots,\underline{z}_{k+\ell-1}$ are linearly independent,
for all large $k$ or certain arbitrarily large $k$.
Our assumptions will involve bounds for the logarithmic quotients
of consecutive linear form evaluations and norms of best approximations, more precisely 
we employ the quantities
\begin{equation} \label{eq:quan}
\sigma_{k}:=\frac{\log \Vert \Theta^{E}\underline{z}_{k+1}\Vert}{\log \Vert \Theta^{E}\underline{z}_{k}\Vert}, \qquad \tau_{k}:=\frac{\log \Vert \underline{z}_{k+1}\Vert}{\log \Vert \underline{z}_{k}\Vert}, \qquad
\nu_{k}:=-\frac{\log \Vert \Theta^{E}\underline{z}_{k}\Vert}{\log \Vert \underline{z}_{k}\Vert}.
\end{equation}
For $m=1$ or $n=1$, similar quantities
regarding quotients consecutive minimal point norms on one hand 
and consecutive approximation qualities on the other hand,
but without taking logarithms,
have recently been studied by Akhunzhanov and Moshchevitin~\cite{akhmosh}.
Some consequences of their work are briefly skteched in Example~\ref{exa} below. 

We have $\sigma_{k}>1, \tau_k >1$ by \eqref{eq:ndef}, 
and $\nu_{k}>\widehat{w}(\Theta)-o(1)\geq n/m-o(1)$ as $k\to\infty$ by \eqref{eq:didit}. Moreover, the upper limit of $\nu_{k}$ as $k\to\infty$
coincides with $w(\Theta)$.
Furthermore
\begin{equation}  \label{eq:einfachi}
\sigma_{k}=\frac{\tau_{k}\nu_{k+1}}{\nu_{k}}, \qquad\qquad \tau_{k}\nu_{k+1}>\nu_{k}, 
\end{equation}
where the right claim follows from \eqref{eq:ndef}. We give more details
on these quantities below Corollary~\ref{itsc}.
We will further use the derived values
\[
\underline{\sigma}:= \liminf_{k\to\infty} \sigma_{k}, \qquad
\overline{\sigma}:= \limsup_{k\to\infty}\sigma_{k}, 
\]
that complement \eqref{eq:tauuuuu} and
are again bounded from below by $1$ and may attain 
the formal value $+\infty$.
Let
\[
\Gamma(\Theta)= 1+\frac{\log(\widehat{w}(\Theta)(\overline{\tau}-1)(\underline{\sigma}-1)+1)}{\log \overline{\tau}},
\]
and
\[
\widetilde{\Gamma}(\Theta)= 1+ \frac{ \log \left(  (\overline{\tau}-1)\left( \widehat{w}(\Theta)(\underline{\sigma}-1)+1-\frac{1}{\underline{\tau} } \right)+1 \right)}{\log \overline{\tau}},
\]
where here and below we take the right limit if $\overline{\tau}=1$.
% CORRECTED: added after submission    oct 18, 2021 !!!!!!!!
Since $\underline{\tau}\geq 1$, for any $\Theta$ we have
\[
\Gamma(\Theta)\leq \widetilde{\Gamma}(\Theta).
\]
Our first result
is the following 

\begin{theorem} \label{siegele}
	Let $m,n\geq 1$ and $\Theta\in\mathbb{R}^{m\times n}$ a good matrix 
	with associated minimal point sequence $(\underline{z}_{k})_{k\geq 1}$.
	If we let $\epsilon>0$ and $k\geq k_{0}(\epsilon)$,
	then the assumption
	\begin{equation}  \label{eq:spez}
	\widehat{w}(\Theta)> \frac{\tau_{k+1}+\tau_{k}^{-1}}{\sigma_{k}-1}+\epsilon
	\end{equation}
	implies that
	$\underline{z}_{k}, \underline{z}_{k+1},\underline{z}_{k+2}$
	are linearly independent. 
	Now assume that 
	\[
	\underline{\sigma}>1, \qquad\qquad
	\overline{\tau}<\infty.
	\]
	If the integer $\ell\geq 1$ satisfies
	\begin{equation}  \label{eq:starkaug}
	\ell< \widetilde{\Gamma}(\Theta),
	\end{equation}
	which is in particular true if $\ell< \Gamma(\Theta)$,
	then for all large indices $k$ the vectors
	\begin{equation} \label{eq:WITT}
	\underline{z}_{k}, \underline{z}_{k+1},\ldots,\underline{z}_{k+\ell-1}
	\end{equation}
	are linearly independent. 
	If we assume that either of the slightly weaker conditions
	\begin{equation}  \label{eq:KOMMT}
	\ell< 1+\min \left\{ \widetilde{\Gamma}(\Theta) \; , \;
	\frac{ \log \left(  (\overline{\tau}-1)\left( \widehat{w}(\Theta)(\overline{\sigma}-1)+1-\frac{1}{\underline{\tau} } \right)+1 \right)}{\log \overline{\tau}} 
	\right\},
	\end{equation}
	or 
	\begin{equation} \label{eq:FT}
	\ell < 1+\min\left\{  \widetilde{\Gamma}(\Theta) \; , \;
	\frac{ \log \left(  (\overline{\tau}-1)\left( \widehat{w}(\Theta)(\underline{\sigma}-1)+1-\overline{\tau}^{-1}  \right)+1 \right)}{\log \overline{\tau}} \right\}
	\end{equation}
	holds, then \eqref{eq:WITT}
	are linearly independent for infinitely many $k$.
	%As soon as the weaker condition 
	%
	%\[
	%\ell< \frac{\sigma-1}{\tau^{\ell-1}}\cdot w(\underline{\theta})+1
	%\]
	%
	%holds, there are infinitely many indices $k$ for which
	%$\underline{z}_{k}, \underline{z}_{k+1},\ldots,\underline{z}_{k+\ell-1}$
	%are linearly independent.
	
	%The same claims hold for infinitely many $k$ when we replace
	%$\underline{\sigma}$ by $\overline{\sigma}$ thorughout.
	% CHECK IF THIS MIGHT BE TRUE !!!!!!!!!!!!!!!!!!!!!!!!!!!!!!
	% !!!!!!!!!!!!!!!!!!!!!!!!!!!!!!!!!!!!!!!!!!!!!!!!
\end{theorem}

\begin{remark} \label{rehhirsch}
	In view of \eqref{eq:diri}, we can relax the
	conditions in all \eqref{eq:spez}-\eqref{eq:FT} by replacing
	$\widehat{w}(\Theta)$ by $n/m$.
\end{remark}

Notice that the right bounds in
the minima in \eqref{eq:KOMMT} and \eqref{eq:FT} 
are obtained by replacing $\underline{\sigma}$ by
$\overline{\sigma}$ and $\underline{\tau}$ by $\overline{\tau}$ 
respectively in $\widetilde{\Gamma}(\Theta)$, and subtracting $1$. 
This indeed relaxes \eqref{eq:starkaug} in both cases.
A weakened version of
the right bound in the minimum in \eqref{eq:KOMMT} with simpler
bound expression is obtained
via replacing $\underline{\tau}$ by $1$, likewise as $\Gamma(\Theta)$ 
arises from $\widetilde{\Gamma}(\Theta)$.

%\begin{remark} \label{himann}
%	It would be nice to provide weaker conditions 
%	than \eqref{eq:starkaug}, \eqref{eq:konk}, for 
%	example replacing $\underline{\sigma}$ with
%	$\overline{\sigma}$ in \eqref{eq:konk},
%	that imply the independence
%	for certain arbitrarily large $k$ only. Unfortunately,
%	our proof seems not to allow this implication.
%\end{remark}

We deduce a corollary.

\begin{corollary} \label{itsc}
	For any good matrix $\Theta\in\mathbb{R}^{m\times n}$ we have
	\[
	\widehat{w}(\Theta)\leq \frac{\overline{\tau}^{n+m}-1- 
		(\overline{\tau}-1)( 1-\frac{1}{\underline{\tau}}) }{
		(\overline{\tau}-1)(\underline{\sigma}-1)}.
	\]
	In particular 
	\[
	\frac{n}{m}\leq \frac{\overline{\tau}^{n+m}-1}{
		(\overline{\tau}-1)(\underline{\sigma}-1)},
	\]
	and $\widehat{w}(\Theta)=\infty$ implies that 
	either $\underline{\sigma} =1$ or $\overline{\tau}=\infty$. 
	% CORRECTED AFTER SUBMISSION  oct 18, 2021  !!!!!!!!!!!!!!!
\end{corollary}

The last assertion applies in particular to the 
very singular vectors belonging to $\textbf{V}_3$ constructed by Moshchevitin~\cite{mosh} for $m=1$. Probably the latter claim $\overline{\tau}=\infty$ is true. Jarn\'ik's~\cite{jarnik} 
% CORRECTED AFTER SUBMISSION  oct 18, 2021  !!!!!!!!!!!!!!!
estimate \eqref{eq:just} implies the ratio $w(\Theta)/\widehat{w}(\Theta)$
tends to infinity with $\widehat{w}(\Theta)$, however for
$\overline{\tau}$ this seems not quite clear. Lemma~\ref{lemur} 
below only contains reverse estimates.

\begin{proof}
	For $\ell=m+n+1$, the vectors \eqref{eq:WITT}
	are clearly linearly dependent, hence the estimate
	\eqref{eq:starkaug} must be false. This is equivalent to the first claim. The weaker second claim then follows from \eqref{eq:diri} and $\underline{\tau}\geq 1$.
\end{proof}

Roughly speaking, Theorem~\ref{siegele} and Corollary~\ref{itsc} tell 
us that if the approximation qualities induced by any two consecutive
best approximations differ significantly, then the norms of certain 
two consecutive best approximations must also increase at some
minimum rate. This relation gets even stronger if $\widehat{w}(\Theta)$
exceeds $n/m$ significantly.

To give some flavor of the strength of the bounds,
if $\underline{\sigma}>1, \overline{\tau}>1$ and the ratio $n/m=:c$
are all fixed, then by \eqref{eq:diri} we satisfy \eqref{eq:starkaug} for an $\ell\geq \log c+d-o(1)$ for some $d$, independent of $n$. For $c,d$ 
not too small this may be of interest. 
If for all large (resp. infinitely many) $k$ we
can improve the trivial lower bound $\underline{\sigma}$ and/or upper bound $\overline{\tau}$ for $\ell-1$ consecutive values $\sigma_{k},\ldots,\sigma_{k+\ell-2}$ and/or $\tau_{k},\ldots,\tau_{k+\ell-2}$, the conditions
\eqref{eq:starkaug} (resp. \eqref{eq:KOMMT} or \eqref{eq:FT}) of Theorem~\ref{siegele}
can be relaxed.
See also Theorem~\ref{rr} below.
A sharp upper estimate for both
in terms of exponents of approximation
is provided in the following lemma.

\begin{lemma} \label{lemur}
	Let $\Theta\in\mathbb{R}^{m\times n}$ a good matrix
	and assume $\widehat{w}(\Theta)<\infty$. 
	Then we have 
	\begin{equation}  \label{eq:unterea}
	1\leq \max\{\underline{\sigma},\underline{\tau}\}\leq \max\{\overline{\sigma},\overline{\tau}\}\leq   \frac{w(\Theta)}{\widehat{w}(\Theta)}\leq
	\frac{m}{n}\cdot w(\Theta).
	\end{equation}
\end{lemma}

%When combined with Corollary~\ref{itsc}, we get
%
%\begin{corollary}  
%	For any good matrix $\Theta\in\mathbb{R}^{m\times n}$ we have the %equivalent claims
%	\[
%	\widehat{w}(\Theta) \leq \frac{ %\left(\frac{w(\Theta)}{\widehat{w}(\Theta)}\right)^{m+n}-1 %}{\frac{w(\Theta)}{\widehat{w}(\Theta)}-1} \cdot %\frac{1}{\underline{\sigma}-1 }\;, \qquad
%	\underline{\sigma}\leq \frac{ %\left(\frac{w(\Theta)}{\widehat{w}(\Theta)}\right)^{m+n}-1 %}{\frac{w(\Theta)}{\widehat{w}(\Theta)}-1} \cdot %\frac{1}{\widehat{w}(\Theta) } +1.
%	\]
%\end{corollary}
%
%The left estimate complements the stronger 
%but conditional bound \eqref{eq:diegute}
%with $h=m+n$, the right complements $\underline{\sigma}\leq %w(\Theta)/\widehat{w}(\Theta)$ from Lemma~\ref{lemur}.
Unfortunately we require a 
non-trivial lower estimate for $\underline{\sigma}$ in our applications.
A generic $\Theta$ satisfies
$w(\Theta)=\widehat{w}(\Theta)=n/m$ and hence induces $\underline{\sigma}=\overline{\sigma}=\underline{\tau}=\overline{\tau}=1$. On the other hand, for a ''typical'' $\Theta$ satisfying $w(\Theta)>n/m$, we 
expect $\underline{\sigma}>1$ (and $\underline{\tau}>1$) or at least $\overline{\sigma}>1$
(and $\overline{\tau}>1$),
see Theorem~\ref{anni} or Theorem~\ref{better01}. 
However, the relation between the exponents $w(\Theta), \widehat{w}(\Theta)$ and the values $\underline{\sigma},\overline{\sigma},\underline{\tau},\overline{\tau}$ 
can be complicated as the next example demonstrates.

\begin{example} \label{exa}
	Let $\theta\in\mathbb{R}$ an extremal number as defined by Roy~\cite{roy}.
	If $m=2, n=1$ and $\Theta=(\theta,\theta^{2})^{t}$, then
	Roy's results in that paper (in particular~\cite[Theorem~5.1]{roy}
	and its proof) imply
	\[
	\overline{\tau}=\overline{\sigma}=\frac{ w(\Theta) }{ \widehat{w}(\Theta) }=\frac{1}{ \frac{\sqrt{5}-1}{2}}=\frac{\sqrt{5}+1}{2}.
	\]
	Possibly also $\underline{\tau}=\underline{\sigma}=(\sqrt{5}+1)/2$,
	however this seems not clear from~\cite[Theorem~5.1]{roy}.
	Similarly for $m=1,n=2$ 
	and $\Theta=(\theta,\theta^{2})$.
	Regardless if this is true, there is identity at least in the third
	inequality in \eqref{eq:unterea}.
	On the other hand,
	for $m=1, n=3$ and $\Theta=(\theta,\theta^{2},\theta^{3})$, the description
	of the associated parametric graph in~\cite{s1} shows that actually
	$\overline{\tau}=1$ and $\underline{\sigma}=1$, even though $w(\Theta)=\sqrt{5}+2>3=\widehat{w}(\Theta)$.
	However, the construction suggests that
	$\overline{\sigma}=w(\Theta)/\widehat{w}(\Theta)=(2+\sqrt{5})/3>1$ in this case. If $n=1$, it seems the construction by Akhunzhanov and Moshchevitin~\cite[Theorem~2]{akhmosh} provides cases
	where $\overline{\tau}>1$ but $\overline{\sigma}=1$, however
	we cannot have $\overline{\tau}=1$ and $\overline{\sigma}>1$ simultaneously
	by a similar argument as in~\cite[Theorem~1]{akhmosh}. Vice versa should be  expected when $m=1$, see also the announced~\cite[Theorem~3]{akhmosh}.
\end{example}

The method in~\cite{mamo} shows that for $m=1$, the 
assumption $\widehat{w}(\Theta)>n$
implies $\overline{\tau}\geq G^{\ast}>1$,
and similarly if $n=1$ then $\overline{\tau}\geq G>1$, with
$G=G(\widehat{w}(\Theta),m)$ and $G^{\ast}=G^{\ast}(\widehat{w}(\Theta),n)$ defined as in~\cite[Theorem~1]{mamo}. The latter $G^{\ast}(\widehat{w}(\Theta),n)$ we denoted by $G_{1,n}(\widehat{w}(\Theta))$ in the rephrased Theorem~\ref{anni} above. 
%
%On the other hand $\tau=w(\underline{\theta})/\widehat{w}(\underline{\theta})$.
%For $n=2$ Sturmian continued fractions $\theta$ satisfy the hypothesis of the %theorem,
%with $\sigma=\tau=(\sqrt{5}-1)/2$ specifically for extremal numbers.
%In the interesting case $n>2$ unfortunately we cannot provide concrete examples
%with $\sigma>1$.  (Note: For general $\underline{\theta}$ they exist
%by Roy Thm)

We continue with a variant of Theorem~\ref{siegele} where we impose a bound
on the logarithmic quotients of the largest by the smallest vector norm
of a set of consecutive best approximations instead.

\begin{theorem} \label{rr}
	Let $m,n,\Theta, (\underline{z}_{k})_{k\geq 1}$ as 
	above and $\ell\geq 3$
	and integer. Let $\sigma^{\prime}>1$ and $\tau^{\prime}\geq 1, \tau^{\ast}\geq 1$	be real numbers and 
	$k$ be a large integer. 
	Assume for $\sigma_{j},\tau_{j}$ defined in \eqref{eq:quan}
	we have
	\begin{equation} \label{eq:assudamm}
	\sigma_{j}\geq \sigma^{\prime}, \quad \tau_{j}\geq \tau^{\prime}, \qquad
	\qquad k\leq j\leq k+\ell-3,
	\end{equation}
	and
	\[
	\tau_{k}\tau_{k+1}\cdots \tau_{k+\ell-2}=\frac{\log \Vert\underline{z}_{k+\ell-1}\Vert}{\log \Vert\underline{z}_{k}\Vert}\leq \tau^{\ast}.
	\]
	%
	%(This implies $\tau^{\ast}\geq \tau^{\prime \ell-1})$. 
	Then with  
	\begin{equation}  \label{eq:Lambdo}
	\Lambda:= 1+\tau^{\prime -1}+\tau^{\prime -2}+\cdots
	+ \tau^{\prime -(\ell-2)}=
	\frac{1-\tau^{\prime 1-\ell}}{1-\tau^{\prime -1}}\leq \ell-1,
	\end{equation}
	if we have 
	\begin{equation}  \label{eq:lammdortn}
	\widehat{w}(\Theta)> \frac{\tau^{\ast}\Lambda-\tau^{\prime}+1}{\tau^{\prime}(\sigma^{\prime}-1)},
	\end{equation}
	then the best approximations $\underline{z}_{k},\underline{z}_{k+1},\ldots,\underline{z}_{k+\ell-1}$ are linearly independent. By rearrangements,
	the same conclusion holds if
	\begin{equation}  \label{eq:lammhier}
	\ell < 1-\frac{ \log \left(1- \frac{\tau^{\prime}-1}{\tau^{\prime}}
		\cdot \frac{\widehat{w}(\Theta)\tau^{\prime}(\sigma^{\prime}-1)+\tau^{\prime}-1}{\tau^{\ast}}    \right)}{\log \tau^{\prime}}.
	\end{equation}
\end{theorem}

Dealing with consecutive minimal points is not too crucial 
in Theorem~\ref{rr}, it can be generalized in a straightforward way
to any increasingly ordered minimal points satisfying similar
relations.
% CORRECTED: new sentence after submission  oct 18, 2021 !!!!!!!!!!
We notice that $\underline{\tau}^{\ell-1}-\epsilon \leq \tau^{\ast}\leq \overline{\tau}^{\ell-1}+\epsilon$ for large $k\geq k_{0}(\epsilon)$.
Again by \eqref{eq:diri} we can 
replace the factor $\widehat{w}(\Theta)$ by the possibly smaller
value $n/m$, to obtain a weaker result that avoids exponents. 
We want to state two weaker but simpler conditions in a corollary.

\begin{corollary}  \label{fehl}
	With the notation of Theorem~\ref{rr}, if
	\begin{equation} \label{eq:susdortn}
	\ell< \frac{\sigma^{\prime}-1}{\tau^{\ast}}\cdot (\widehat{w}(\Theta)+\tau^{\prime}-1)+1,
	\end{equation}
	which is in particular true if
	\begin{equation}  \label{eq:sushier}
	\ell< \frac{\sigma^{\prime}-1}{\tau^{\ast}}\cdot \widehat{w}(\Theta)+1,
	\end{equation}
	then the best approximations $\underline{z}_{k},\underline{z}_{k+1},\ldots,\underline{z}_{k+\ell-1}$ are linearly independent. Thus for any 
	good $\Theta\in\mathbb{R}^{m\times n}$
	we have
	\begin{equation}  \label{eq:succula}
	\widehat{w}(\Theta) \leq \frac{ (m+n)\tau^{\ast} }{\sigma^{\prime}-1}.
	\end{equation}
\end{corollary}

\begin{proof}
	Implication from \eqref{eq:susdortn}
	follows from \eqref{eq:lammdortn} when estimating $\Lambda\leq \ell-1$,
	then using $\tau^{\prime}\geq 1$ gives the weaker condition \eqref{eq:sushier}. Finally
	for $\ell=m+n+1$ the linear independence conclusion fails, 
	hence the
	reverse inequality of \eqref{eq:sushier} must hold, giving the claim \eqref{eq:succula}.
\end{proof}

The last claim \eqref{eq:succula} holds for general 
$\Theta\in \mathscr{G}_{h}$ with 
the factor $m+n$ replaced by $h$. For example
any $\Theta=\underline{\theta}\in \mathscr{H}_{3}^{1,n}$ satisfies
$\widehat{w}(\Theta)\leq 3\tau^{\ast}/(\sigma^{\prime}-1)$.
Weaker claims by replacing
$\tau^{\ast}$ by $\overline{\tau}^{\ell-1}$ can be stated, for $\ell=3$
this is implied by \eqref{eq:spez}.
We provide another linear independence criterion for $\ell=3$ 
complementary to \eqref{eq:spez},
where we make hypotheses on 
two consecutive approximation qualities, reflected by $\nu_{k}, \nu_{k+1}$.

\begin{theorem} \label{202}
	Keep the notation of Theorem~\ref{rr} and let $\epsilon>0$. Assume that 
	$k\geq k_{0}(\epsilon)$ is large and as in \eqref{eq:quan} let
	\[
	\nu_{k}= -\frac{\log \Vert \Theta^{E}\underline{z}_{k}\Vert}{\log \Vert\underline{z}_{k}\Vert}, \quad \nu_{k+1}= -\frac{\log \Vert \Theta^{E}\underline{z}_{k+1}\Vert}{\log \Vert\underline{z}_{k+1}\Vert}, 
	\quad \tau_{k}=\frac{\log \Vert\underline{z}_{k+1}\Vert}{\log \Vert\underline{z}_{k}\Vert},  \quad \tau_{k+1}=\frac{\log \Vert\underline{z}_{k+2}\Vert}{\log \Vert\underline{z}_{k+1}\Vert}.
	\]
	If at least one of the three conditions
	\begin{equation} \label{eq:nr1}
	\nu_{k}-\epsilon > \frac{\tau_{k}\tau_{k+1}+1}{\sigma_{k}-1} 
	\end{equation}
	or
	\begin{equation}  \label{eq:nr3}
	\nu_{k}+1+\epsilon < \tau_{k}(\nu_{k+1}-\tau_{k+1})
	\end{equation}
	or
	\begin{equation} \label{eq:nr2}
	\tau_{k}(\tau_{k}\nu_{k+1}-\nu_{k})\widehat{w}(\Theta)^{2}
	-\nu_{k}\widehat{w}(\Theta)
	-\tau_{k}\nu_{k}\nu_{k+1}>\epsilon,
	\end{equation}
	holds, then $\underline{z}_{k}, \underline{z}_{k+1},\underline{z}_{k+2}$
	are linearly independent. 
\end{theorem}

The second condition \eqref{eq:nr3} is just slightly stronger than
\[
\nu_{k+1}-\tau_{k+1} > \frac{\nu_{k}}{\tau_{k}} - \varepsilon.
\]
Up to subtraction of (the possibly large) $\tau_{k+1}$ this resembles \eqref{eq:einfachi}. Observe that
the bracket expression in \eqref{eq:nr2} is positive by \eqref{eq:einfachi}. 
Again we can write $n/m$ in place of $\widehat{w}(\Theta)$ in \eqref{eq:nr2},
and also $\nu_{i}\geq n/m-o(1)$ as $i\to\infty$ by Dirichlet's Theorem
\eqref{eq:diri}
and $\tau_{i}>1$ by \eqref{eq:ndef}. 
The third hypothesis \eqref{eq:nr2} holds in particular 
if $\nu_{k+1}$ is sufficiently large and $\nu_{k}/\tau_{k}<\widehat{w}(\Theta)^{2}$,
the latter being true if
$\nu_{k}/\tau_{k}<(n/m)^{2}$. 
Theorem~\ref{202} should be viewed as a ''local result'', the fact that
three consecutive minimal points
are linearly independent
for infinitely many $k$ often follows without further assumption, 
as recalled in Section~\ref{linin}.

We finish this section by remarking that
some considerations concerning the simultaneous approximation case $n=1$
can be extracted from
Davenport and Schmidt~\cite{davsh}, see in particular Lemma~5 in that paper.
We believe that the underlying arguments can be adapted to get more insight. 
For $n=1, m=2$ recall Lagarias' result from~\cite{lag0} quoted
in Section~\ref{r3}. See also Section~\ref{outline} below.

\subsection{The Veronese curve} \label{se23}

We now consider $n=1$ and more specifically
that $\Theta=\underline{\theta}\in\mathbb{R}^{m\times 1}$ consists
of successive powers of a number, that is $(\underline{\theta},1)$
lies on the twisted Veronese curve  $V_{n}:=\{(\theta^{n},\ldots,\theta^{2},\theta,1):\theta\in\mathbb{R}\}$
with coordinates in reverse order.
We will sporadically identify the
vector $\Theta$ with its 
first coordinate $\theta\in\mathbb{R}$ in the sequel.
Then the scalar
product of the minimal
points $\underline{z}_{k}$ with $\Theta^{E}=(\underline{\theta},1)\in\mathbb{R}^{n+1}$ 
can be interpreted as an integer polynomial of degree at most $n$ 
evaluated at $\theta$. We denote by $P_{k}$ this polynomial that 
realizes
$\Theta^{E}\underline{z}_{k}= P_{k}(\theta)$, 
call $P_{k}$ best approximation polynomial associated to the pair $\theta,n$
and write $H(P_{k})$
for $\Vert \underline{z}_{k}\Vert$ and call it height of $P_{k}$. According to \eqref{eq:ndef},
the sequence $(P_{k})_{k\geq 1}$ satisfies
\[
H(P_{1})<H(P_{2})< \cdots, \qquad |P_{1}(\theta)|>|P_{2}(\theta)|>\cdots.
\] 
The classical notation for the linear form exponents of 
approximation in this case is
\[
w(\Theta)=w_{n}(\theta), \qquad \widehat{w}(\Theta)=\widehat{w}_{n}(\theta).
\]
The claims of previous sections clearly apply to the 
special case of the Veronese curve.
We first highlight a consequence of Theorem~\ref{anni}
when combined with a result of Sprind\v{z}uk~\cite{sprindzuk}.

\begin{definition}
	Let $\mathscr{G}_{h,n}\subseteq \mathscr{G}_{h}^{1,n}$ be the points
	in $\mathscr{G}_{h}^{1,n}$ of the form $(\theta^n,\theta^{n-1},\ldots,\theta)$.
\end{definition}

\begin{corollary}
	Let $n>h\geq 3$ be integers. Then 
	the set $\mathscr{G}_{h,n}$
	has $1$-dimensional Lebesgue measure $0$ (Hausdorff dimension
	less than $1$) and contains no 
	vector with algebraic $\theta$.
\end{corollary}

%\begin{proof}
%	The metric claims follow from Theorem~\ref{better} and
%	Theorem~\ref{4t} upon metric results by Sprind\v{z}uk~\cite{sprindzuk} and
%	a refinement due to Bernik~\cite{bernik}, respectively. The claim on %algebraic vectors
%	follows directly from Theorem~\ref{better} and Theorem~\ref{4t} again.
%\end{proof}

The first metric claim is valid for the much larger class
of so-called extremal curves, including any smooth curve
that is properly curved. We only want to refer here to
a very general result by 
%R. Baker~\cite{baker} and 
Kleinbock and Margulis~\cite{bv}. Conrete
bounds for the Hausdorff dimensions of $\mathscr{G}_{h,n}$
for $h<n$ can be derived from combining \eqref{eq:diegute} with
the metric result of Bernik~\cite{bernik},
for $h=3$ we get that $\mathscr{G}_{3,n}$
has dimension at most $(n+1)/(n^{2}-n+1)=O(n^{-1})$ for $n\geq 2$,
smaller than $1$ if $n>2$.

Our proof of the next result requires the Veronese
curve setting. We adapt the notation concerning $\sigma, \tau, \nu$ from Section~\ref{cri}.

\begin{theorem}  \label{cc}
	Let $n\geq 1$ and a real number $\theta$ not algebraic of degree $\leq n$
	be given and consider the best approximation
	polynomials $(P_{k})_{k\geq 1}$ associated
	to $\theta,n$. Assume for any large $k$
	the polynomials $P_{k}, P_{k+1}$ have no common factor and we have
	\begin{equation} \label{eq:assuandamm}
	\underline{\nu}:= \liminf_{k\to\infty} -\frac{\log \vert P_{k}(\theta)\vert}{\log H(P_{k})}>2n-1. %\qquad w_{n}(\theta)<\infty.  
	\end{equation}
	%
	%in particular
	%
	%\[
	%H(P_{k})^{-w_{2}}< \vert P_{k}(\theta)\vert < H(P_{k})^{-w_{1}}, \qquad\qquad %k\geq k_{0}.
	%\]
	%
	Then 
	\begin{equation} \label{eq:showw}
	\underline{\sigma}:=\liminf_{k\to\infty}\frac{\log \vert P_{k+1}(\theta)\vert}{\log \vert P_{k}(\theta)\vert}\geq \frac{\underline{\nu}-n+1}{n}, \qquad \overline{\tau}:=\limsup_{k\to\infty}\frac{\log H(P_{k+1})}{\log H(P_{k})}\leq \frac{w_{n}(\theta)}{\widehat{w}_{n}(\theta)},
	\end{equation}
	and hence if the integer $\ell\geq 1$ satisfies
	\begin{equation} \label{eq:lhe}
	\ell< \frac{\log\left( \frac{(\underline{\nu}-2n+1)(w_{n}(\theta)-\widehat{w}_{n}(\theta))}{n}+1\right)}{\log(w_{n}(\theta)/\widehat{w}_{n}(\theta))}+1\leq \frac{\log\left( \frac{(\underline{\nu}-2n+1)(w_{n}(\theta)-n)}{n}+1\right)}{\log(w_{n}(\theta)/n)}+1,
	\end{equation}
	then for every large $k$ the polynomials $P_{k}, P_{k+1},\ldots,P_{k+\ell-1}$
	are linearly independent.
\end{theorem}

The condition \eqref{eq:lhe} can be slightly relaxed, see
the connection between $\Gamma$ and $\widetilde{\Gamma}$ in
Section~\ref{cri}. Moreover variants with relaxed conditions
and conclusions for infinitely many $k$ only can be readily derived.
We state some other remarks.

\begin{remark}
	We may also state a stronger version than \eqref{eq:showw} 
	involving the accodingly defined quantity $\underline{\tau}$, 
	analogously to \eqref{eq:starkaug}.
	Note that $\underline{\nu}$ and $w_{n}(\theta)$ are related by
	\[
	\underline{\nu}= \liminf_{k\to\infty} -\frac{\log \vert P_{k}(\theta)\vert}{\log H(P_{k})}\leq \limsup_{k\to\infty} -\frac{\log \vert P_{k}(\theta)\vert}{\log H(P_{k})}=w_{n}(\theta).
	\]
	It may be true that
	$n=2$ and $\theta$ any extremal number~\cite{roy} 
	provide a non-trivial equality case, compare this with 
	Example~\ref{exa} above.
	Unfortunately, it is not clear how to 
	link $\underline{\nu}$ with $\widehat{w}_{n}(\theta)$.
\end{remark}

\begin{remark}
	The coprimality condition is satisfied as soon as $w_{n-1}(\theta)<\underline{\nu}$,
	as then the polynomials $P_{k}$ are irreducible of degree precisely $n$
	for every large $k$, so in 
	particular if $w_{n-1}(\theta)\leq 2n-1$.
	In case of $w_{n}(\theta)>w_{n-1}(\theta)$ and $\widehat{w}_{n}(\theta)>n$,
	due to Lemma~\ref{lemur} and \cite[Theorem~2.2]{bugschlei} 
	we can estimate
	\[
	\max\{\underline{\sigma},\underline{\tau}\} \leq \max\{\overline{\sigma},\overline{\tau}\} \leq \frac{w_{n}(\theta)}{\widehat{w}_{n}(\theta)}\leq  \frac{n-1}{\widehat{w}_{n}(\theta)-n}.
	\]
\end{remark}

If $w_n(\theta)\geq \delta n$ for $\delta>2$, 
then we may choose $\ell \gg \log n$ again with 
an implied constant independent from $n$.
The condition \eqref{eq:showw} of the theorem states that {\em all} best approximation polynomials
induce very small evaluations at $\theta$, with the natural exponent
$n$ replaced by some value $>2n-1$.
We could similarly derive variants of Theorem~\ref{cc}
in the spirit of Theorem~\ref{rr} for the
Veronese curve under assumption of \eqref{eq:assuandamm}. We only
want to state an improvement of Theorem~\ref{202} in the Veronese curve case.

\begin{theorem} \label{36}
	Let $\theta$ be a transcendental real number and $n\geq 2$ be an integer
	and denote by $(P_{j})_{j\geq 1}$ the sequence of best approximation
	polynomials associated to $\theta, n$. Let $\epsilon>0$.
	Assume $k$ is a large index and that
	$P_{k}$ and $P_{k+1}$ are coprime. As in \eqref{eq:quan} let
	\[
	\nu_{k}=-\frac{\log \vert P_{k}(\theta)\vert}{\log H(P_{k})}, \qquad
	\nu_{k+1}=-\frac{\log \vert P_{k+1}(\theta)\vert}{\log H(P_{k+1})}, \qquad 
	\tau_{k+1}=\frac{\log H(P_{k+2})}{\log H(P_{k+1})}.
	\]
	Assume that $\nu_{k}>2n-1$ and
	\begin{itemize}
		\item either
		the relation
		%
		%\begin{equation} \label{eq:hh}
		%\frac{\gamma(\alpha+1)}{\alpha+\gamma}>n
		%\end{equation}
		%
		%	and 
		%
		\begin{equation}  \label{eq:sech}
		(\chi_{k}^{2}\nu_{k+1}-\chi_{k} \nu_{k})\widehat{w}_{n}(\theta)^2-\nu_{k}\widehat{w}_{n}(\theta)-\chi_{k}\nu_{k}\nu_{k+1}>0, \qquad \chi_{k}=\frac{\nu_{k}-n+1}{n},
		\end{equation}
		\item or 
		\begin{equation} \label{eq:sech2}
		(\nu_{k+1}-\tau_{k+1})\frac{  \nu_{k}-n+1}{\nu_{k}} > n.
		\end{equation}
	\end{itemize}
	holds. Then $P_{k},P_{k+1},P_{k+2}$ are linearly independent.
\end{theorem}

It can be verified that upon $\nu_{k}>2n-1$ the condition \eqref{eq:sech}
relaxes \eqref{eq:nr2} and \eqref{eq:sech2}
relaxes \eqref{eq:nr1}, 
when we trivially estimate $\tau_{k}$ by $1$ in \eqref{eq:nr2} resp.  \eqref{eq:nr1}.
Finally
we want to generalize Theorems~\ref{siegele},~\ref{rr} 
to certain sets of
polynomials derived from consecutive best approximation polynomials
by multiplication with integer polynomials of small degree ($\leq d$). 
Sets of this type have been of interest in~\cite{s2}, 
where it was shown that certain mild linear independence conditions 
imply good upper bounds 
on the classical exponent $\widehat{w}_{n}(\theta)$.
The main obstacle for our method in this setting
is that for $d>0$ the new polynomials may have small evaluations at $\theta$ as well.
For this reason the quantity $w_{d}(\theta)$ will occur. We agree
on the notation $w_{0}(\theta)=0$.

\begin{theorem} \label{konse}
	Let $n\geq 1$ be an integer and $\theta$ be a real number and
	let $(P_{k})_{k\geq 1}$ be the best approximation polynomial
	sequence associated to $n,\theta$. 
	Define $\underline{\sigma}\geq 1, \overline{\tau}\geq 1$  as in \eqref{eq:showw} 
	and let $\ell\geq 3, d\geq 0$ be other integers satisfying $(d+1)\ell\leq n+d+1$.
	Assume the equivalent conditions
	\begin{equation} \label{eq:tut}
	\ell< \frac{\frac{\widehat{w}_{n}(\theta)(\underline{\sigma}-1)\underline{\tau}}{\overline{\tau}^{\ell-1}(w_{d}(\theta)+1)}+1}{d+1}
	\quad \Longleftrightarrow \quad \widehat{w}_{n}(\theta)> \frac{[(d+1)\ell-1)](w_{d}(\theta)+1)\overline{\tau}^{\ell-1}}{(\underline{\sigma}-1)\underline{\tau}}
	\end{equation}
	hold. Define the sets of polynomials
	\[
	\mathscr{A}_{j}(T)= \{ P_{j}(T), TP_{j}(T),\ldots,T^{d}P_{j}(T)\}, \qquad\qquad j\geq 1.
	\]
	Then for all large indices $k$, the set
	$\mathscr{B}_{k}:= \mathscr{A}_{k}\cup \mathscr{A}_{k+1}\cup \cdots \mathscr{A}_{k+\ell-1}$ 
	consisting of $(d+1)\ell$ polynomials of degree at most $n+d$, is linearly independent. 
\end{theorem}

As before we may replace $\widehat{w}_{n}(\theta)$ by $n$
in \eqref{eq:tut} to get weaker claims.
The choice $d=0$ leads to criterion \eqref{eq:sushier}
of Corollary~\ref{fehl} in the special 
case of the Veronese curve upon identifying
$\tau^{\ast}$ with $\overline{\tau}^{\ell-1}$, see also the remarks below Theorem~\ref{rr}.
Some improvements in the spirit of Theorem~\ref{rr} can be obtained
upon certain refinements in the proof, we do not state them explicitly.
We see that if $\underline{\sigma}>1, \overline{\tau}$ are fixed and $w_{d}(\theta)\ll d$ then 
for large $n$ again we
have that $\mathscr{B}_{k}$ in the theorem is linearly
independent for $\ell$ up to some value $\gg \log n-2\log d$. If $d$
is fixed as well and $w_{d}(\theta)<\infty$, again  for
large $n$ the claim is true for $\ell$ up to $\gg \log n$.

\section{Proofs}

\subsection{Siegel's Lemma} \label{sigl}

A crucial ingredient of our proofs is Siegel's Lemma. The most
effective variant for our purposes is reproduced below. See also
Davenport and Schmidt~\cite[Theorem~3]{ds2} proved in Section~11
of their paper.

\begin{lemma}[Siegel's Lemma]
	Consider a system of linear equations
	\[
	B\underline{x}=\underline{0},
	\]
	where $B\in\mathbb{Z}^{m\times u}$ 
	is a matrix with $m$ rows and $u$ columns, and $u>m$. Assume the rows
	are linearly independent, i.e. the matrix has rank $m$. Then there
	is a solution $\underline{x}=(x_{1},\ldots,x_{u})^{t}\in\mathbb{Z}^{u}\setminus\{\underline{0}\}$ 
	of norm $\Vert x\Vert\leq (u-m)V^{1/(u-m)}$,
	for $V$ the maximum modulus of the $m\times m$-subdeterminants of
	the matrices formed	by $m$ columns of $B$. 
\end{lemma}

We point out that the occurring determinants can be estimated up to a factor
$\ll_{m} 1$
by the product of the column norms by Hadamard's inequality. Moreover
the standard version of Siegel's Lemma with $\Vert\underline{x}\Vert\ll_{m} \max_{i,j} \vert b_{i,j}\vert^{m/(u-m)}$, where $b_{i,j}$ are the entries of $B$, follows directly.
We will apply the following modified version.

\begin{corollary} \label{scor}
	Let $B^{\prime}$ be any integer $m\times u$-matrix of rank $s<u$ 
	(possibly with $m>u$). Then the system $B^{\prime}\underline{x}=\underline{0}$ has a solution
	$\underline{x}\in\mathbb{Z}^{u}\setminus\{\underline{0}\}$ with
	$\Vert \underline{x}\Vert\ll_{m} V^{\prime 1/(u-s)}\leq V^{\prime}$
	% exponent corrected after 2nd submission !!!!! July 10, 2020
	where again $V^{\prime}$ is the maximum
	absolute value of the $s\times s$-subdeterminants of $B^{\prime}$.  
\end{corollary}

\begin{proof}
	We form
	a new auxiliary matrix $B$ by taking
	any $s$ linearly independent rows from $B^{\prime}$. and define
	$V$ for $B$ as above. 
	We can apply Siegel's Lemma in the above version to $B$ and
	obtain that $B\underline{x}^{t}=\underline{0}^{t}$ has a solution
	$\underline{x}\in\mathbb{Z}^{u}\setminus\{0\}$
	of norm $\Vert \underline{x}\Vert\leq V^{1/(u-s)}\leq V$. 
	However, since the potential
	other $m-s$ lines of $B^{\prime}$ are each a linear combination
	of the $s$ linearly independent lines of $B$ (since $B^{\prime}$ has rank $s$), 
	clearly $\underline{x}$ is also
	a solution to the original system $B^{\prime}\underline{x}=\underline{0}$.
	Finally, since every $s\times s$ submatrix of
	$B$ is also a submatrix of $B^{\prime}$, clearly $V\leq V^{\prime}$. 
\end{proof}

\subsection{Outline of proofs} \label{outline}

The proofs of all main results of the paper below basically follow
the same line. 
We assume a putative linear dependence equation
\[
a_{1}\underline{r}_{1}+a_{2}\underline{r}_{2}+\cdots +a_{v}\underline{r}_{v}=\underline{0},
\]
for $\underline{r}_{j}=\underline{z}_{i_j}$ certain best approximations, mostly consecutive, associated to
$\Theta$ and suitable $v$. 
From Siegel's Lemma
in the form of Corollary~\ref{scor} and Hadamard's estimate we derive upper
bounds for $\Vert a\Vert=\max \vert a_{j}\vert$ in terms of the norms
$\Vert \underline{r}_{j}\Vert$. 
The above identity implies
\[
a_{1}\Theta^{E}\underline{r}_{1}+\cdots+a_{v}\Theta^{E}\underline{r}_{v}=
\Theta^{E}(a_{1}\underline{r}_{1}+a_{2}\underline{r}_{2}+\cdots +a_{v}\Theta\underline{r}_{v})=\Theta^{E}\cdot \underline{0}= \underline{0}.
\]
Now if the maximum of the terms, say $\Vert \Theta^{E}\underline{r}_{1}\Vert$,
is reasonably larger than all other expressions $\Vert \Theta^{E}\underline{r}_{i}\Vert, i\neq 1$,
using the bounds for the coefficients we get a contradiction by
triangular inequality, unless $a_{1}=0$ which must be considered separately.
We finish this short section with the proof of the auxiliary lemma.
Observe there is a typographical difference between 
different quantities $\epsilon$ and $\varepsilon$.

\begin{proof}[Proof of Lemma~\ref{lemur}]
	Let $\epsilon>0$. Let $\underline{z}_{k}$ be a best approximation of large index $k$.
	Then by definition of $w(\Theta)$ we have
	\[
	\nu_{k}=-\frac{\log \Vert \Theta^{E}\underline{z}_{k}\Vert}{\log \Vert \underline{z}_{k}\Vert}\leq
	w(\Theta)+\epsilon.
	\]
	Now let $\varepsilon=2 \widehat{w}(\Theta)\epsilon>\widehat{w}(\Theta)\epsilon>0$ and
	\[
	X:= \Vert \underline{z}_{k}\Vert^{w(\Theta)/\widehat{w}(\Theta)-\varepsilon}.
	\]
	By definition of $\widehat{w}(\Theta)$ the system
	\[
	\Vert \underline{z}\Vert\leq X, \qquad 
	\Vert \Theta^{E}\underline{z}\Vert\leq X^{-\widehat{w}(\Theta)+\epsilon}
	\]
	has a solution $\underline{z}\in\mathbb{Z}^{n+m}\setminus \{\underline{0}\}$
	if $k$ was chosen large enough. Note that the right estimate
	is not satisfied for $\underline{z}=\underline{z}_{k}$ 
	by choice of $\varepsilon$. Thus by definition of best approximations \eqref{eq:ndef} we infer
	$X\geq \Vert \underline{z}_{k+1}\Vert$,
	showing the estimate for $\overline{\tau}$ as $\epsilon$ and thus
	$\varepsilon$ can be chosen arbitrarily small.
	
	For the estimate for $\sigma$ again start with any large $k$ and 
	observe that
	%
	%\[
	%\nu_{k}= -\frac{\log | \underline{z}_{k}\cdot \underline{\theta}|}{\log \Vert %\underline{z}_{k}\Vert}\geq
	%\widehat{w}(\underline{\theta})-\epsilon.
	%\]
	%
	a slight modification of the proof of the estimate for 
	$\overline{\tau}$ above (writing $\nu_{k}$ in place of $w(\Theta)$) 
	shows that
	\begin{equation}  \label{eq:hopla}
	\tau_{k}=\frac{\log \Vert \underline{z}_{k+1}\Vert}
	{\log \Vert \underline{z}_{k}\Vert}\leq \frac{\nu_{k}}{\widehat{w}(\Theta)}+\varepsilon.
	\end{equation}
	See the proof of Theorem~\ref{202} below for a concise justification.
	Observe further that
	\[
	\nu_{k+1}=-\frac{\log \Vert \Theta^{E}\underline{z}_{k+1}\Vert}{\log \Vert \underline{z}_{k+1}\Vert}\leq
	w(\Theta)+\epsilon
	\]
	holds. Combining these properties yields
	\[
	\frac{\log \Vert\Theta^{E}\underline{z}_{k+1}\Vert}{\log \Vert \Theta^{E}\underline{z}_{k}\Vert}=
	-\frac{\log \Vert \Theta^{E}\underline{z}_{k+1}\Vert}{\log \Vert \underline{z}_{k+1}\Vert}\cdot \frac{\log \Vert \underline{z}_{k+1}\Vert}{\log \Vert \underline{z}_{k}\Vert}
	\cdot -\frac{\log \Vert \underline{z}_{k}\Vert}{\log \Vert \Theta^{E}\underline{z}_{k}\Vert}\leq (w(\Theta)+\epsilon)(\frac{\nu_{k}}{\widehat{w}(\Theta)}+\varepsilon)\nu_{k}^{-1}.
	\]
	The claim follows as $\epsilon,\varepsilon\to 0$. The most right inequality in \eqref{eq:unterea} now comes from \eqref{eq:diri}.
\end{proof}

\subsection{Proof of Theorem~\ref{better01}}

In the proofs below any appearing $\epsilon_{i}$ will be positive
but arbitrarily small.
We first observe the following easy, auxiliary result. Notice again the typographical difference between $\epsilon$ and $\varepsilon$ in the proof.

\begin{proposition}  \label{propo}
	Assume $w(\Theta)<\infty$. Then if for every $t$
	we choose any $\underline{v}_t\in \mathbb{Z}^{m+n}$
	with $\Vert \underline{v}_{t}\Vert \ll \Vert \underline{z}_{t}\Vert^{o(1)}$
	as $t\to\infty$, we have
	\[
	\Vert \Theta^{E} \underline{v}_t\Vert \geq \Vert\underline{z}_{t}\Vert^{-o(1)}.
	\]
\end{proposition}

\begin{proof}
	We may clearly assume $\Vert \underline{v}_t\Vert$ tends 
	to infinity with $t$. 
	Then by definition of $w(\Theta)$ for large $t\geq t_{0}$
	we have
	\[
	\Vert \Theta^{E} \underline{v}_t\Vert \geq \Vert \underline{v}_t\Vert^{-2w(\Theta)}.
	\]
	By assumption, for any $\epsilon>0$ and large $t\geq t_{1}(\epsilon)$, 
	we have  
	$\Vert \underline{v}_t\Vert \leq \Vert \underline{z}_t\Vert^{\epsilon}$.
	For given $\varepsilon>0$,
	with $\epsilon:=\varepsilon/(2w(\Theta))$, we conclude
	\[
	\Vert \Theta^{E} \underline{v}_t\Vert \geq \Vert \underline{v}_t\Vert^{-2w(\Theta)}\geq \Vert \underline{z}_t\Vert^{-\varepsilon}, \qquad t\geq \max\{ t_{0},t_{1}\}.
	\]
	As $\varepsilon$ can be arbitrarily small, the claim follows.
\end{proof}

Let $m\geq 1, n\geq 1$ and $1\leq h\leq m+n$ be fixed and $\Theta\in\mathbb{R}^{m\times n}$
belong to $\textbf{V}_h$. Consider sets of consecutive
minimal vectors $\underline{z}_{t}, \underline{z}_{t+1},\ldots, \underline{z}_{t+h-1}$ for large $t$ as
in the definition of $\textbf{V}_h$. 
To shorten notation, let
\[
\mathscr{F}_t= \scp{\underline{z}_{t}, \underline{z}_{t+1},\ldots, \underline{z}_{t+h-1}}_{\mathbb{R}} \subseteq \mathbb{R}^{m+n}, \qquad t\geq 1,
\]
be the vector spanned by the $\underline{z}_i$.
By assumption
some integer vector $\underline{v}_t$ of norm $\Vert \underline{v}_{t}\Vert \ll \Vert \underline{z}_{t}\Vert^{o(1)}$ 
as $t\to\infty$ lies in the lattice
$\mathscr{F}_t\cap \mathbb{Z}^{m+n}$. 
For each $t\geq 1$ let 
\[
\{ \underline{y}_{1}, \ldots, \underline{y}_w \} \subseteq 
\{ \underline{z}_{t}, \underline{z}_{t+1},\ldots, \underline{z}_{t+h-1} 
\}, \qquad 1\leq w\leq h,
\]
where $w=w(t)$ and the $\underline{y}_{i}=\underline{y}_{i}(t)$ depend on $t$ as well,
be a linearly independent set spanning the same space
\[
\scp{\underline{y}_{1}, \underline{y}_{2},\ldots, \underline{y}_{w}}_{\mathbb{R}}=\mathscr{F}_t,
\]
in other words a vector space basis of $\mathscr{F}_t$. 
Assume the norms $\Vert\underline{y}_{i}\Vert$ are
naturally increasingly ordered.
Let 
$Y:= \Vert \underline{z}_{t}\Vert\in\mathbb{N}$. Now 
by linear independence of the $\underline{y}_i$ and
since $Y\cdot \underline{v}_{t}$ obviously lies in the lattice $\mathscr{F}_t\cap \mathbb{Z}^{m+n}$
as well, we have a one-dimensional solution space to the identity
\begin{equation} \label{eq:Frti}
a_{0}Y\underline{v}_{t}+a_{1}\underline{y}_{1}+a_{2}\underline{y}_{2}+\cdots +a_{w}\underline{y}_{w}=\underline{0},
\end{equation}
in $\underline{a}=(a_0,\ldots,a_w)$, and any 
non-zero solution has $a_0\neq 0$. 
Since we deal with integer vectors, 
the integral solutions to \eqref{eq:Frti} 
form a one-dimensional lattice in $\mathbb{Z}^{w+1}$.
In other words, we have a unique generator
solution $\underline{a}=(a_0,\ldots,a_w)$ that is a
primitive (i.e. largest common divisor equals $1$) integer vector with $a_0>0$, and all other integer solutions to \eqref{eq:Frti} are integer multiples of it. In particular our $\underline{a}$ minimizes
the norm among all non-zero integer solutions. Fix this $\underline{a}$
in the sequel,
for simplicity we do not invent new notation for it.
If $w(\Theta)=\infty$ the claim is trivial, so
we can assume $w(\Theta)<\infty$. Then 
for given $\varepsilon>0$, by Proposition~\ref{propo}
we have $\Vert \Theta^{E} \underline{v}_t\Vert \geq \Vert\underline{z}_{t}\Vert^{-\varepsilon}= Y^{-\varepsilon}$ for large $t$. Since $a_0\neq 0$,	thus  
\[
\Vert a_{0}\Theta^{E}Y\underline{v}_{t}\Vert=
|a_{0}|\cdot Y\cdot \Vert \Theta^{E}\underline{v}_{t}\Vert\geq Y\cdot
\Vert \Theta^{E}\underline{v}_{t}\Vert\geq Y^{1-\varepsilon}, \qquad t\geq t_{0}.
\]
On the other hand, by \eqref{eq:Frti} we have
\begin{align*}
&a_{0}\Theta^{E}(Y\underline{v}_{t})+a_{1}\Theta^{E}\underline{y}_{1}+a_{2}\Theta^{E}\underline{y}_{2}+\cdots+a_{w}\Theta^{E}\underline{y}_{w} \\ &=\Theta^{E}(a_{0}Y\underline{v}_{t}+a_{t}\underline{y}_{1}+a_{2}\underline{y}_{2}+\cdots +a_{w}\underline{y}_{w})=\underline{0},
\end{align*}
so for $t\geq t_{0}$ we infer  % HERE TO GO ON !!!!!!!!!!
\begin{equation}  \label{eq:torr2T}
S:=\Vert a_{1}\Theta^{E}\underline{y}_{1}+a_{2}\Theta^{E}\underline{y}_{2}+
\cdots+a_{w}\Theta^{E}\underline{y}_{w}\Vert =\Vert a_{0}Y \Theta^{E}\underline{v}_{t}\Vert\geq Y^{1-\varepsilon}.
\end{equation}
%
%Let $
%X:= \Vert \underline{z}_{t+h-1}\Vert$.
Let 
\[
\alpha_{i} = \alpha_{i,t}= \frac{\log \Vert \underline{z}_{t+i}\Vert}{\log \Vert \underline{z}_{t}\Vert}> 1, \qquad  1\leq i\leq h-1,
\]
so that %$X=\Vert \underline{z}_{t+h-1}\Vert=Y^{\alpha_{h-1} }$
%and 
$\Vert \underline{z}_{t+1}\Vert=Y^{\alpha_{1} }$.
Now equation \eqref{eq:Frti} can be written $B\underline{a}=\underline{0}$
for $B$ the integer matrix whose $w+1$ columns consist of the vectors
$Y\underline{v}_{t}, \underline{y}_{1}, \underline{y}_{2},\ldots, \underline{y}_{w}$ respectively
and
$\underline{a}=(a_{0},a_{1},a_{2},\ldots,a_{w})$. 
By assumption $B$ has rank $w$.
Any $w\times w$ subdeterminant of $B$ can by Hadamard's inequality
be estimated up to a factor $\ll_{m,n} 1$ 
by the product of the column norms. Since $\Vert Y \underline{v}_{t}\Vert = Y \cdot \Vert \underline{v}_{t}\Vert \leq Y^{1+\varepsilon}$,
Siegel's Lemma in form of Corollary~\ref{scor} and the 
minimality of $\Vert \underline{a}\Vert$
thus imply that our generator solution $\underline{a}$
to \eqref{eq:Frti} satisfies
\begin{align} \label{eq:pandaT}
\Vert \underline{a}\Vert&= \max|a_{j}|\ll_{m,n} 
\max\{ \Vert \underline{y}_{1}\Vert, Y^{1+\varepsilon}\}\cdot 
\Vert \underline{y}_{2}\Vert\cdot \cdots \Vert \underline{y}_{w}\Vert \\
&\leq
Y^{1+\varepsilon}\cdot \Vert \underline{z}_{t+1}\Vert\cdot \cdots \Vert \underline{z}_{t+h-1}\Vert= Y^{1+\alpha_{1}+\cdots+\alpha_{h-1}+\varepsilon},  \nonumber
\end{align}
no matter whether $\underline{y}_1=\underline{z}_t$ or not.
Now $\Vert \underline{y}_i\Vert \geq \Vert \underline{z}_{t}\Vert$ and \eqref{eq:ndef}, \eqref{eq:didit} imply
\begin{equation} \label{eq:pantaT}
\max_{1\leq i\leq w} \Vert \Theta^{E}\underline{y}_{i}\Vert \leq
\Vert \Theta^{E}\underline{z}_{t}\Vert\ll_{m,n} \Vert \underline{z}_{t+1}\Vert^{-\widehat{w}(\Theta)+\epsilon_{1} }=
\Vert \underline{z}_{t}\Vert^{-\alpha_{1}\widehat{w}(\Theta)+\epsilon_{2} }=
Y^{-\alpha_{1}\widehat{w}(\Theta)+\epsilon_{2}}.
\end{equation}
Since the moduli of the scalar products are decreasing
according to \eqref{eq:ndef},
combining \eqref{eq:pandaT}, \eqref{eq:pantaT}
yields that the sum in \eqref{eq:torr2T} can be estimated from above by
\[
S\leq w
\Vert\underline{a}\Vert\cdot \Vert\Theta^{E}\underline{z}_{t}\Vert\leq h
\Vert\underline{a}\Vert\cdot \Vert\Theta^{E}\underline{z}_{t}\Vert\ll_{m,n} Y^{1+\alpha_{1}+\cdots+\alpha_{h-1}-\alpha_{1}\widehat{w}(\Theta)+\epsilon_{3}}.
\]
Since $Y\to\infty$ 
as $t\to\infty$ and $\varepsilon$ can be arbitrarily small, 
making up for the multiplicative factor with arbitrarily small quantity and
combining with the lower estimate \eqref{eq:torr2T} yields 
\[
\alpha_{1}(1-\widehat{w}(\Theta))+\alpha_{2}+\cdots+\alpha_{h-1} 
\geq -\epsilon_{4}.
\]
Hence, no matter if $\widehat{w}(\Theta)>1$ or not, 
as $t\to\infty$ we conclude 
\begin{equation}  \label{eq:egalist}
\widehat{w}(\Theta) \leq 1+\frac{ \alpha_{2}+\cdots+\alpha_{h-1}}{\alpha_{1}}+ \epsilon_{5}.
\end{equation}
By Lemma~\ref{lemur} we see that $\alpha_{i+1}/\alpha_{i}=\tau_{t+i}\leq w(\Theta)/\widehat{w}(\Theta)+\epsilon_{6}$ for all $i$ under consideration. Hence the right hand side
can be estimated via
\[
\widehat{w}(\Theta) \leq 1+\frac{w(\Theta)}{\widehat{w}(\Theta)}+ \left(\frac{w(\Theta)}{\widehat{w}(\Theta)}\right)^2 +\cdots+
\left(\frac{w(\Theta)}{\widehat{w}(\Theta)}\right)^{h-2}+\epsilon_{7},
\]
the claim \eqref{eq:diegute} follows as $\epsilon_{7}$ will be
arbitrarily small. 
The claim \eqref{eq:nettz} is clear from
\eqref{eq:egalist} as well. 
Finally upon $\Theta\in \textbf{SV}_h$, we can choose $t$
so that the quotient $\alpha_{2}/\alpha_{1}=\tau_{t+1}$ is arbitrarily close to $\underline{\tau}$, and
\eqref{eq:hallg} follows. The proof is finished.

\subsection{Proofs of Section~\ref{cri}}
Similar ideas are employed to prove the results of Section~\ref{cri}. 
The notation
$\lfloor x\rfloor$ indicates the largest integer smaller than or equal to $x\in\mathbb{R}$.

\begin{proof}[Proof of Theorem~\ref{siegele}]
	Let $m,n,\Theta$ as in the theorem. We first show the claim
	involving \eqref{eq:starkaug}. Hence let
	$\ell\geq 3$ be another fixed integer to be specified later. 
	Assume the opposite, that is $\underline{z}_{k},\ldots,\underline{z}_{k+\ell-1}\in\mathbb{Z}^{n+m}$ 
	are linearly dependent for some large $k$, which
	we consider fixed in the sequel. 
	Derive $\underline{y}_1, \ldots, \underline{y}_w$, with $w=w(t) $,
	a subset that forms a basis of $\mathscr{F}_k=\scp{\underline{z}_{k},\ldots,\underline{z}_{k+\ell-1}}_{\mathbb{R}}$, labeled with increasing norms, very similar to the proof
	of Theorem~\ref{better01}. Notice $w\leq \ell-1$ here.
	First assume we may choose the $\underline{y}_i$ so that
	$\underline{z}_{k}\notin \{ \underline{y}_1, \ldots, \underline{y}_w \}$.
	Let $\underline{y}_0= \underline{z}_{k}$.
	For simplicity,	let
	\[
	Y_{j}:= \Vert \underline{y}_{j}\Vert, \qquad\qquad\qquad 0\leq j\leq w.
	\]
	Put $Z:= \lfloor  Y_{1}/ Y_{0}\rfloor\geq 1$ 
	so that $\Vert Z\underline{y}_{0}\Vert\asymp \Vert \underline{y}_{1}\Vert$. 
	Let $B$ be the $(m+n)\times (w+1)$-matrix with first column
	$Z\underline{y}_{0}$ and $j$-th column
	$\underline{y}_{j-1}$ for $2\leq j\leq w+1$.
	Consider the system
	\begin{equation} \label{eq:vanlin}
	B\underline{a}=a_{0}Z\underline{y}_{0}+a_{1}\underline{y}_{1}+
	\cdots+a_{w}\underline{y}_{w} = \underline{0},
	\end{equation}
	for $\underline{a}=(a_0,\ldots,a_{w})$. By the same argument
	as in Theorem~\ref{better01},
	there is a unique primitive 
	integer vector $\underline{a}$ with $a_0>0$
	that generates the one-dimensional lattice of all integer solutions.
	Recall $\tau_{j}$ from \eqref{eq:quan}, that
	clearly satisfy $\tau_{j}>1$ for all $j$ by \eqref{eq:ndef}. 
	Since $\Vert Z\underline{y}_{0}\Vert=ZY_0\ll Y_{1}$ and $B$ has rank $w$
	and by the minimality of $\Vert \underline{a}\Vert$,
	Siegel's Lemma in form of Corollary~\ref{scor} implies that
	our primitive integer solution vector has entries
	\begin{equation} \label{eq:togo}
	\Vert \underline{a}\Vert=\max_{0\leq j\leq w} \vert a_{j}\vert\ll_{m,n} Y_{1}Y_{2}\cdots Y_{w}\leq \Vert \underline{z}_{k+1}\Vert \cdot\Vert \underline{z}_{k+2}\Vert \ldots \Vert\underline{z}_{k+\ell-1}\Vert \leq Y_0^{R},
	\end{equation}
	where we have put
	\[
	R:=\tau_{k}+\tau_{k}\tau_{k+1}+\cdots +\tau_{k}\tau_{k+1}\cdots \tau_{k+\ell-2}.
	\]
	Let $\epsilon>0$. We may assume $k$ was chosen large enough that
	\begin{equation} \label{eq:ee}
	\sigma_{k}=\frac{\log \Vert \Theta^{E}\underline{z}_{k+1}\Vert}{\log \Vert \Theta^{E}\underline{z}_{k}\Vert}>\underline{\sigma}-\epsilon.
	\end{equation}
	From $a_{0}\neq 0$ we further infer
	\[
	\Vert \Theta^{E}a_{0}Z\underline{y}_{0}\Vert= \vert a_{0}\vert Z\cdot 
	\Vert \Theta^{E}\underline{y}_{0}\Vert\geq
	Z\Vert \Theta^{E}\underline{y}_{0}\Vert>0.	
	\]
	Since in view of \eqref{eq:vanlin} we have
	\[
	\Theta^{E}a_{0}Z\underline{y}_{0}+\Theta^{E}a_{1}\underline{y}_{1}+\cdots+\Theta^{E}a_{w}\underline{y}_{w} =\Theta^{E}\cdot \underline{0}=\underline{0},	
	\]
	we infer
	\begin{equation}  \label{eq:lhs4}
	S:=\Vert \Theta^{E}a_{1}\underline{z}_{1}+\cdots+\Theta^{E}a_{w}\underline{y}_{w}\Vert
	= \Vert a_{0}Z\Theta^{E}\underline{y}_{0}\Vert
	\geq Z\Vert \Theta^{E}\underline{y}_{0}\Vert.	
	\end{equation}
	On the other hand, by \eqref{eq:ndef} and \eqref{eq:ee} we have
	\[
	\max_{1\leq j\leq w} \Vert \Theta^{E}\underline{y}_{j}\Vert =\Vert \Theta^{E}\underline{y}_{1}\Vert \leq \Vert \Theta^{E}\underline{z}_{k+1}\Vert\leq \Vert \Theta^{E}\underline{z}_{k}\Vert^{\underline{\sigma}-\epsilon}= \Vert \Theta^{E}\underline{y}_{0}\Vert^{\underline{\sigma}-\epsilon}.
	\]
	%
	%(Remark: Here the problems indicated in Remark~\ref{himann} occur 
	%in the setting there since possibly $\underline{y}_{1}\neq %\underline{z}_{k+1}$). % CHECK !!!!!!!!! 
	Consequently we can estimate $S\leq\ell \Vert \underline{a}\Vert\cdot \Vert \Theta^{E}\underline{y}_{1}\Vert\ll_{m,n} Y_0^{R}\Vert \Theta^{E}\underline{y}_{0}\Vert^{\underline{\sigma}-\epsilon}$
	and hence
	\[
	Y_0^{R}\Vert \Theta^{E}\underline{y}_{0}\Vert^{\underline{\sigma}-\epsilon}\gg_{m,n} Z\Vert \Theta^{E}\underline{y}_{0}\Vert,
	\]
	or equivalently
	\begin{equation}  \label{eq:dem}
	\Vert \Theta^{E}\underline{y}_{0}\Vert \geq  Z^{1/(\underline{\sigma}-1)+\epsilon_{1} }Y_0^{-R/(\underline{\sigma}-1)+\epsilon_2 }=
	Y_0^{(\tau_{k}-1)/(\underline{\sigma}-1)+\epsilon_{1} }Y_0^{-R/(\underline{\sigma}-1)+\epsilon_{2} },
	\end{equation}
	%
	%As this holds for arbitrarily large $k$ we infer
	%
	%\[
	%w_{n}(\theta)\leq \frac{\ell \tau^{\ell-1}}{\sigma-1}.
	%	\]
	for $\epsilon_{1}>0, \epsilon_{2}>0$ small variations of $\epsilon$.
	On the other hand, since $\underline{y}_{0}$ is a minimal point, by \eqref{eq:didit} we infer
	\[
	\Vert \Theta^{E}\underline{y}_{0}\Vert \leq Y_0^{-\tau_{k}\widehat{w}(\Theta)+\epsilon_3}.
	\]
	% 
	%(Here we may write exponent $-\tau_{s}\widehat{w}(\underline{\theta})$ to %get the stronger claim of Remark~\ref{rehhirsch}.)
	Combining with \eqref{eq:dem} yields
	\begin{equation} \label{eq:ewins}
	\tau_{k}\widehat{w}(\Theta)\leq \frac{R-\tau_{k}+1}{\underline{\sigma}-1}+\epsilon_{4}.
	\end{equation}
	We may assume $k$ is large enough that
	\[
	\underline{\tau}-\epsilon\leq \tau_{i}\leq \overline{\tau}+\epsilon ,\qquad\qquad i\geq k-1.
	\]
	The value $R$ can be bounded
	\[
	R\leq\tau_{k}(1+\overline{\tau}+\overline{\tau}^2 + \cdots+ \overline{\tau}^{\ell-2}) + \epsilon_{5} =
	\tau_{k}\frac{ \overline{\tau}^{\ell-1}-1}{\overline{\tau}-1}+\epsilon_{5}.
	\]
	Here and below we always take the limit if $\overline{\tau}=1$.
	% CORRECTION: added after submission, oct 18, 2021 !!!!!!!!!!!
	From \eqref{eq:ewins} we infer
	\begin{align}
	\widehat{w}(\Theta)\tau_{k}(\underline{\sigma}-1) + \tau_{k}-1\leq\; 
	\tau_{k}\frac{ \overline{\tau}^{\ell-1}-1}{\overline{\tau}-1}+\epsilon_{6}. \label{eq:e1}
	\end{align}
	Estimating
	$\tau_{k}\geq \underline{\tau}-\epsilon$ we get
	\[
	\overline{\tau}^{\ell-1}-1 \geq (\overline{\tau}-1)\left( (\widehat{w}(\Theta)(\underline{\sigma}-1)+1-\frac{1}{\underline{\tau} } \right) +\epsilon_{7}.
	\]
	%
	%We may assume the bracket expression is positive, otherwise
	%we obtain stronger bounds.
	Solving for $\ell$, we see that 
	\begin{equation} \label{eq:TITZ}
	\ell > 1+ \frac{ \log \left(  (\overline{\tau}-1)\left( (\widehat{w}(\Theta)(\underline{\sigma}-1)+1-\frac{1}{\underline{\tau} } \right)+1 \right)}{\log \overline{\tau}} +\epsilon_{8}= \widetilde{\Gamma}(\Theta)+\epsilon_{8}.
	\end{equation}
	Taking the contrapositive yields the claim
	involving \eqref{eq:starkaug} of the theorem.
	The specialization \eqref{eq:spez} follows since if $\ell=3$, then 
	we have	$R=\tau_{k}(1+\tau_{k+1})$
	and also can take $\sigma_{k}$ instead of $\underline{\sigma}-\epsilon$
	because the two consecutive minimal points $\underline{z}_{k+1}, \underline{z}_{k+2}$
	are always linearly independent. Then a short calculation indeed verifies \eqref{eq:spez}. 
	
	Now assume otherwise that we cannot choose the $\underline{y}_i$ so that
	$\underline{z}_{k}\notin \{ \underline{y}_1, \ldots, \underline{y}_w \}$.
	This means that $\scp{\underline{z}_{k+1}, \ldots, \underline{z}_{k+\ell-1}}_{\mathbb{R}}$
	span a proper subspace of $\mathscr{F}_k=\scp{\underline{z}_{k}, \underline{z}_{k+1}, \ldots, \underline{z}_{k+\ell-1}}_{\mathbb{R}}$. Hence the set  
	$\{\underline{z}_{k+1}, \ldots, \underline{z}_{k+\ell-1}\}$
	is linearly dependent as well. Thus upon index shift
	$k+1$ becoming $k$, we have reduced the problem
	from $\ell$ to $\ell-1$. By an inductive argument, upon
	accordingly redefining $k$, we must 
	end up at some point where can assume the property $\underline{z}_{k}\notin \{ \underline{y}_1, \ldots, \underline{y}_w \}$
	is satisfied. Thus
	we infer \eqref{eq:e1} for some $\ell^{\prime}\leq \ell-1$ replacing
	$\ell$ in the right hand side, and with the same arguments 
	finally end up at the stronger condition
	\begin{equation} \label{eq:bajkus}
	\ell \geq \ell^{\prime} +1 > 1+\widetilde{\Gamma}(\Theta) +\epsilon_{8}
	\end{equation}
	for linear dependence. 
	Again taking the contrapositive, we conclude that condition 
	\eqref{eq:starkaug} suffices in any case for linear independence. 
	
	Finally we prove the last claims. We start with large $k$ that satisfy
	$\sigma_{k}\geq \overline{\sigma}-\epsilon$ instead of
	\eqref{eq:ee} and assume $\underline{z}_{k}, \underline{z}_{k+1}, \ldots, \underline{z}_{k+\ell-1}$ are linearly dependent. 
	Then proceeding as above and again
	distinguishing the
	two cases $\underline{z}_{k}\notin \{ \underline{y}_1, \ldots, \underline{y}_w \}$ and $\underline{z}_{k}\in \{ \underline{y}_1, \ldots, \underline{y}_w \}$, we get the reverse estimates as in
	the right and left bound in \eqref{eq:KOMMT} for $\ell$, respectively.
	Hereby we use \eqref{eq:bajkus} for the latter case.
	Again taking the contrapositive yields the claim. Similarly,
	we can assume $\tau_k\geq \overline{\tau}-\epsilon$
	by $\overline{\tau}$ for certain arbitrarily large
	$k$. This leads to a replacement of $\underline{\tau}$ by
	$\overline{\tau}$ in \eqref{eq:TITZ}, and the analogous arguments
	yield the sufficient condition \eqref{eq:FT}.
\end{proof}

\begin{remark}
	Assume the space $\mathscr{F}_k$ has dimension $w<\ell-1$
	strictly for all large $k$. Then
	we can readily refine the bound $\widetilde{\Gamma}(\Theta)$ in \eqref{eq:TITZ} for $\ell$, as we may take the smaller value 
	\[
	\tilde{R}(w)=\tau_{k}\tau_{k+1}\cdots \tau_{k+\ell-1-w}+\tau_{k}\tau_{k+1}\cdots \tau_{k+\ell-w}+\cdots+\tau_{k}\tau_{k+1}\cdots \tau_{k+\ell-2}
	\]
	in place of $R$.
	Similarly if we assume the property
	for infinitely many $k$. This applies in particular to $\Theta\in \textbf{V}_h$ when we identify $h=w$.    
\end{remark}

The proof of Theorem~\ref{rr} works very similarly, we just estimate
the coefficients in \eqref{eq:vanlin} with Siegel's Lemma in a slightly
different way. Again it is understood that $\epsilon_i$ will all be 
positive but arbitrarily small as the initial $\epsilon>0$ tends
to $0$.

\begin{proof}[Proof of Theorem~\ref{rr}] 
	For $\ell>0$ an integer to be fixed later and large $k$ again 
	assume the opposite that
	$\underline{z}_{k},\ldots,\underline{z}_{k+\ell-1}\in\mathbb{Z}^{n+m}$ 
	are linearly dependent. Define $\underline{y}_i$ 
	and $Y_i$, $0\leq i\leq w$, for $w=w(t)\leq \ell-1$, 
	as in Theorem~\ref{siegele}. Then for the same reasons, 
	again \eqref{eq:vanlin} induces
	a primitive integer vector $\underline{a}=(a_0,\ldots,a_w)$
	with $a_{0}>0$ that generates the lattice of all integer solutions. Again first assume $\underline{z}_{k}\notin \{ \underline{y}_1, \ldots, \underline{y}_w \}$ and take
	$Z$ as in Theorem~\ref{siegele}. 
	For simplicity put $X:= Y_{w}=\Vert \underline{y}_w\Vert$.
	Write \eqref{eq:vanlin} again as a system
	$B\cdot \underline{a}= \underline{0}$
	with $\underline{a}= (a_{0},\ldots,a_{w})^{t}$
	and 
	$B$ the $(m+n)\times (w+1)$ integer matrix of deficient rank
	$w$ whose columns are the vectors $Z\underline{y}_{0},\underline{y}_{1},\ldots,\underline{y}_{w}$. By \eqref{eq:ndef} and
	Siegel's Lemma and Hadamard's estimate, bounding the column norms
	via the assumption on $\tau^{\prime}$, since $w\leq \ell-1$
	we can estimate
	\begin{equation}  \label{eq:jogern}
	\Vert\underline{a}\Vert \ll_{m,n}
	Y_{w}Y_{w-1} \cdots Y_{1}\ll  X^{\Lambda}, 
	\end{equation}
	with
	\[
	\Lambda=1+\tau^{\prime -1}+\tau^{\prime -2}+\cdots
	+ \tau^{\prime -(\ell-2)}=
	\frac{1-\tau^{\prime -(\ell-1)}}{1-\tau^{\prime -1}},
	\]
	as in \eqref{eq:Lambdo}.
	By assumption we have
	\begin{equation} \label{eq:eqobn}
	\sigma_{k}=\frac{\log \Vert \Theta^{E}\underline{z}_{k+1}\Vert}{\log \Vert \Theta^{E}\underline{z}_{k}\Vert}\geq \sigma^{\prime}.
	\end{equation}
	Now since $a_{0}\neq 0$ again we have
	$\Vert a_{0}Z\Theta^{E}\underline{y}_{0}\Vert= \vert a_{0}\vert Z\cdot \Vert \Theta^{E}\underline{y}_{0}\Vert\geq
	Z\Vert \Theta^{E}\underline{y}_{0}\Vert$.
	As in the proof of Theorem~\ref{siegele}, in view of \eqref{eq:vanlin}
	we infer
	\begin{equation}  \label{eq:lhs}
	\Vert a_{1}\Theta^{E}\underline{y}_{1}+\cdots+
	a_{w}\Theta^{E}\underline{y}_{w}\Vert
	= \Vert a_{0}Z\Theta^{E}\underline{y}_{0}\Vert
	\geq Z\Vert \Theta^{E}\underline{y}_{0}\Vert.	
	\end{equation}
	On the other hand, by \eqref{eq:eqobn} we infer
	\[
	\max_{j\geq 1} \Vert \Theta^{E}\underline{y}_{j}\Vert =\Vert \Theta^{E}\underline{y}_{1}\Vert \leq \Vert \Theta^{E}\underline{y}_{0}\Vert^{\sigma^{\prime}-\epsilon},
	\]
	so by \eqref{eq:jogern} the left hand side in \eqref{eq:lhs} is at most $\ell \Vert \underline{a}\Vert\cdot \Vert \Theta^{E}\underline{y}_{1}\Vert^{\sigma^{\prime}-\epsilon}\ll_{m,n} X^{\Lambda}\Vert \Theta^{E}\underline{y}_{0}\Vert^{\sigma^{\prime}-\epsilon}$.
	Combining gives
	\[
	X^{\Lambda}\Vert \Theta^{E}\underline{y}_{0}\Vert^{\sigma^{\prime}-\epsilon}\gg_{m,n} Z\Vert \Theta^{E}\underline{y}_{0}\Vert,
	\]
	or %since $s\geq k$ 
	\[
	X^{\Lambda}\gg_{m,n} Z\Vert \Theta^{E}\underline{y}_{0}\Vert^{-(\sigma^{\prime}-1-\epsilon)}.
	\]
	Now by assumption $\tau_{k}\geq \tau^{\prime}\geq 1$, thus $Z\gg Y_{0}^{\tau^{\prime}-1}$ and 
	\[
	\Vert \Theta^{E}\underline{y}_{0}\Vert \geq
	Y_{0}^{ \frac{\tau^{\prime}-1}{\sigma^{\prime}-1}+\epsilon_{1} } \cdot X^{-\frac{\Lambda}{\sigma^{\prime}-1}+\epsilon_{1} }.
	%\qquad 	\mu=\frac{\Lambda}{\sigma^{\prime}-1}>0,
	% CORRECTED AFTER SUBMISSION  18 oct 2021,    !!!!!!!!!!!!!!!!!!!!
	\]
	On the other hand by assumption
	\[
	X= Y_w\leq  \Vert \underline{z}_{k+\ell-1}\Vert \leq \Vert \underline{z}_{k}\Vert^{\tau^{\ast}}=Y_{0}^{\tau^{\ast}},
	\]
	inserting gives
	\begin{equation} \label{eq:compare}
	\Vert \Theta^{E}\underline{y}_{0}\Vert \geq
	Y_{0}^{ \frac{\tau^{\prime}-1}{\sigma^{\prime}-1} - \frac{\tau^{\ast} \Lambda}{\sigma^{\prime}-1}+\epsilon_{2} }.
	% CORRECTED AFTER SUBMISSION  18 oct 2021,    !!!!!!!!!!!!!!!!!!!!
	\end{equation}
	On the other hand, from Dirichlet's Theorem \eqref{eq:didit} we infer
	\[
	\Vert \Theta^{E}\underline{y}_{0}\Vert \leq Y_0^{-\tau^{\prime}\widehat{w}(\Theta)+\epsilon_{3} }.
	\]
	As all $\epsilon_{i}$ can be made arbitrarily small,
	combining with \eqref{eq:compare} yields
	\[
	\widehat{w}(\Theta)\leq \frac{\tau^{\ast}\Lambda-\tau^{\prime}+1}{\tau^{\prime}(\sigma^{\prime}-1)}+\epsilon_{4}.
	\]
	Taking the contrapositive and as $\epsilon_{4}$ can be arbitrarily small shows that \eqref{eq:lammdortn} indeed
	implies the linear independence of $\underline{z}_{k},\ldots,\underline{z}_{k+\ell-1}$. 
	Inserting for $\Lambda$ from \eqref{eq:Lambdo},
	condition \eqref{eq:lammdortn}
	can be rearranged to \eqref{eq:lammhier}.
	Finally, if the assumption $\underline{z}_{k}\notin \{ \underline{y}_1, \ldots, \underline{y}_w \}$ does not hold, we reduce it to this
	case precisely as in the last paragraph of the proof of Theorem~\ref{better01}.
\end{proof}

\begin{proof}[Proof of Theorem~\ref{202}]
	Assume otherwise
	$\underline{z}_{k}, \underline{z}_{k+1},\underline{z}_{k+2}$
	are linearly dependent so that we have an identity
	\[
	a_{k}\underline{z}_{k}+a_{k+1}\underline{z}_{k+1}+a_{k+2}\underline{z}_{k+2}
	=\underline{0},
	\]
	with integers $a_{k}, a_{k+1},a_{k+2}$ not all $0$. We have $a_{k}\neq 0$ since
	$\underline{z}_{k+1}, \underline{z}_{k+2}$ are linearly independent
	for every $k$, see
	Section~\ref{intro}.
	Upon the first condition we 
	proceed as in the proof of Theorem~\ref{siegele}. Note that
	with $Y_{0}:=\Vert \underline{z}_{k}\Vert$ 
	we have $\Vert \underline{z}_{k+1}\Vert\cdot \Vert \underline{z}_{k+2}\Vert=Y_{0}^{\tau_{k}+\tau_{k}\tau_{k+1}}$.
	On the one hand with $Z= \Vert \underline{z}_{k+1}\Vert/\Vert \underline{z}_{k}\Vert\rfloor\asymp Y_0^{\tau_{k}-1}$ the linear dependence of
	$Z\underline{z}_{k},\underline{z}_{k+1},\underline{z}_{k+2}$ very
	similarly as in \eqref{eq:dem} yields
	\[
	\Vert \Theta^{E}\underline{z}_{k}\Vert \gg Z^{1/(\sigma_{k}-1)} Y_0^{-(\tau_{k}+\tau_{k}\tau_{k+1})/(\sigma_{k}-1)}
	\gg Y_0^{-\mu}, \qquad
	\mu=\frac{\tau_{k}\tau_{k+1}+1}{\sigma_{k}-1}>0.
	\]
	On the other hand by definition
	\[
	\Vert \Theta^{E}\underline{z}_{k}\Vert= Y_0^{-\nu_{k}}.
	\]
	Combining yields $\nu_{k}\leq (\tau_{k}\tau_{k+1}+1)/(\sigma_{k}-1)+o(1)$
	as $k\to\infty$. Thus assuming the reverse inequality
	\eqref{eq:nr1}, we cannot have linear dependence for large $k$.
	The second condition \eqref{eq:nr3} is 
	equivalent to \eqref{eq:nr1} 
	via identity \eqref{eq:einfachi}, which reads
	\begin{equation} \label{eq:halbe7}
	\sigma_{k}= \frac{\log \Vert \Theta^{E}\underline{z}_{k+1}\Vert}{\log \Vert \Theta^{E}\underline{z}_{k}\Vert}= \frac{\tau_{k}\nu_{k+1}}{\nu_{k}}>1,
	\end{equation}
	after a short rearrangement (upon modifying $\epsilon$).
	
	For the conclusion from the third hypothesis 
	we verify the sufficient condition \eqref{eq:spez} i.e.
	\begin{equation} \label{eq:volle7}
	\widehat{w}(\Theta)>\frac{\tau_{k+1}+\tau_{k}^{-1}}{\sigma_{k}-1}
	+\epsilon.
	\end{equation}
	We bound
	$\tau_{k+1}$ from above. We claim for any $\epsilon_{1}>0$ we have 
	\begin{equation}  \label{eq:este7}
	\tau_{k+1}\leq \frac{\nu_{k+1}}{\widehat{w}(\Theta)}+\epsilon_{1},
	\qquad\qquad k\geq k_{0}(\epsilon_{1}).
	\end{equation}
	Inserting for $\tau_{k+1}$ and $\sigma_k$ from
	\eqref{eq:este7} and \eqref{eq:halbe7} in \eqref{eq:volle7}, we derive
	the third criterion \eqref{eq:nr2} after a short rearrangement.
	Hereby we use $\tau_{k}\nu_{k+1}-\nu_{k}>0$ by \eqref{eq:halbe7}.
	We are left to verify
	\eqref{eq:este7}, which we may do for index $k$ instead of
	$k+1$ for simplicity. However, this estimate 
	follows again from Dirichlet's Theorem, similar to \eqref{eq:didit}.
	Assume \eqref{eq:este7} fails (for index $k$ instead of $k+1$).
	Then for suitable small $\epsilon_{2}>0$ (in dependence of $\epsilon_{1}$ above) that we can let tend to $0$ as $\epsilon_{1}\to 0$
	and	the parameter $X=\Vert\underline{z}_{k}\Vert^{\tau_{k}-\epsilon_{2}}$, we have
	\[
	X<\Vert \underline{z}_{k+1}\Vert, \qquad \Vert \Theta^{E}\underline{z}_{k}\Vert= \Vert \underline{z}_{k}\Vert^{-\nu_{k} } > X^{-\widehat{w}(\Theta)+\epsilon_{3} },
	\]
	with some fixed small modification $\epsilon_{3}>0$ of $\epsilon_{2}$.
	Hence, since 
	$\underline{z}_{k}, \underline{z}_{k+1}$
	are consecutive minimal points, the system
	$\Vert\underline{z}\Vert\leq X$ and 
	$\Vert \Theta^{E}\underline{z}\Vert<X^{-\widehat{w}(\Theta)+\epsilon_{3} }$ would have
	no solution in an integer vector $\underline{z}\in\mathbb{Z}^{n+m}\setminus\{\underline{0}\}$.
	This obviously contradicts
	the definition of $\widehat{w}(\Theta)$ as $k\to\infty$ and thereby $X\to\infty$. 
\end{proof}

\subsection{Proofs for the Veronese curve}
The improvements for the Veronese curve rely on the following estimate
based on a variation of Liouville's inequality from~\cite{bugschlei}.

\begin{lemma} \label{nazur}
	Let $n\geq 1$ be an integer and $theta$ real and not algebraic of 
	degree at most $n$. Let $(P_{k})_{k\geq 1}$ be the associated
	best approximation polynomial sequence. Let $\epsilon>0$.
	Assume for some $k\geq k_{0}(\epsilon)$ we have that $P_{k}, P_{k+1}$ have no common factor and that
	$|P_{k}(\theta)|=H(P_{k})^{-\nu_{k}}$ for some $\nu_{k}>2n-1$. Then we have
	\begin{equation}  \label{eq:jup}
	\tau_{k}=\frac{\log H(P_{k+1})}{\log H(P_{k})}\geq \frac{\nu_{k}-n+1}{n}-\epsilon.
	\end{equation}
\end{lemma} 

\begin{proof}
	As a direct consequence of~\cite[Lemma~3.1]{bugschlei}, if
	$\theta$ is any real number and $P,Q$ are coprime 
	polynomials of
	degree at most $n$ and $H(Q)>H(P)$, we have
	\[
	\max\{ |P(\theta)|, |Q(\theta)|\} \gg_{n} H(P)^{-n+1}H(Q)^{-n}.
	\]
	Application to best approximation polynomials $P=P_{k}$ and $Q=P_{k+1}$ yields
	\[
	H(P_{k})^{-\nu_{k}}=|P_{k}(\theta)|= \max\{ |P_{k}(\theta)|, |P_{k+1}(\theta)|\} \gg_{n} H(P_{k})^{-n+1}H(P_{k+1})^{-n}.
	\]
	The claim follows after minor rearrangements.
\end{proof}

We remark that the estimate \eqref{eq:jup} is known to be sharp if $n=2$ and $\theta$ is
a Sturmian continued fraction (see~\cite[Theorem~3.1]{buglau}) or any extremal number~\cite{roy}.

\begin{proof}[Proof of Theorem~\ref{cc}]
	We need to show \eqref{eq:showw}, the bound on $\ell$ then follows essentially as a special case of condition
	\eqref{eq:starkaug} from Theorem~\ref{siegele}. For the left inequality we
	use Lemma~\ref{nazur}. Let $\epsilon>0$ and $k$ be large.
	Write
	\[
	\sigma_{k}=\frac{\log \vert P_{k+1}(\theta)\vert}{\log \vert P_{k}(\theta)\vert}=
	\frac{\tau_{k}\nu_{k+1}}{\nu_{k}}
	\]
	where
	\[
	\nu_{k}=-\frac{\log \vert P_{k}(\theta)\vert}{\log H(P_{k})},\quad \nu_{k+1}=-\frac{\log \vert P_{k+1}(\theta)\vert}{\log H(P_{k+1})},\quad \tau_{k}=
	\frac{\log H(P_{k+1})}{\log H(P_{k})}.
	\]
	By assumption $\nu_{k+1}\geq \underline{\nu}-\epsilon$. Moreover $\tau_{k}$
	can be bounded in terms of $\nu_{k}$ by Lemma~\ref{nazur} via
	\[
	\tau_{k}=\frac{\log H(P_{k+1})}{\log H(P_{k})}\geq \frac{\nu_{k}-n+1}{n}-\epsilon.
	\]  
	%
	%This according estimate $\beta/\alpha\geq (\alpha-n+1)/(n\alpha)$ 
	%remains valid if $\alpha=\infty$.
	Combining yields that 
	\[
	\sigma_{k}\geq
	\frac{\underline{\nu}(\nu_{k}-n+1)}{n\nu_{k}}-\varepsilon.
	\]
	Now by assumption $\nu_{k}\geq \underline{\nu}-\epsilon$ as well, 
	and letting $\epsilon\to 0$ we see that the expression is minimized if $\nu_{k}=\underline{\nu}$ which gives the lower bound $(\underline{\nu}-n+1)/n$
	of the theorem for $\underline{\sigma}$.
	%
	%
	%
	%Since $P_{k-1}, P_{k}$ are both irreducible they are coprime, so we may apply %Liouville inequality.
	%Since $P_{k}$ is a best approximation polynomial
	%we have $H(P_{k-1})\leq H(P_{k})$ and $\vert P_{k-1}(theta)\vert >\vert %P_{k}(theta)\vert$. Thus Liouville's inequaltiy yields
	%
	%\[
	%\vert P_{k-1}(theta)\vert = \max\{ \vert P_{k-1}(theta)\vert, \vert %P_{k}(theta)\vert\} \gg_{n} H(P_{k-1})^{-n}  H(P_{k})^{-n}\cdot H(P_{k-1})\geq  %H(P_{k})^{-2n+1},
	%\]
	%
	%hence 
	%
	%\[
	%\liminf_{k\to\infty}-\frac{\log H(P_{k})}{\log \vert P_{k-1}(theta)\vert} \geq %\frac{1}{2n-1}.
	%\]
	%
	%We conclude
	%
	%\[
	%\sigma=\liminf_{k\to\infty}\frac{\log \vert P_{k}(theta)\vert}{\log \vert %P_{k-1}(theta)\vert}\geq  \liminf_{k\to\infty} -\frac{\log \vert %P_{k}(theta)\vert}{\log H(P_{k})}\cdot \liminf_{k\to\infty}-\frac{\log %H(P_{k})}{\log \vert P_{k-1}(theta)\vert} \geq \frac{w}{2n-1}.
	%\]
	%
	The right estimate for $\overline{\tau}$ is just
	\eqref{eq:unterea}. 
	%Finally note
	%that the
	%factor $n$ can be replaced by $w$ in view of our assumption %\eqref{eq:assuandamm}
	%(instead of Dirichlet's Theorem) to obtain the second claim after
	%a short calculation. 
\end{proof}

\begin{proof}[Proof of Theorem~\ref{36}]
	%Since $P_{k+1},P_{k+2}$ are linearly independent, we may assume $a_{k}\neq 0$
	%is any putative vanishing linear combination %$a_{k}P_{k}+a_{k+1}P_{k+1}+a_{k+2}P_{k+2}\equiv 0$. 
	For the first condition we combine criterion \eqref{eq:nr2}
	from Theorem~\ref{202} with \eqref{eq:jup} from
	Lemma~\ref{nazur}.  Observe that
	when expanding the expression in \eqref{eq:nr2}
	as a quadratic function in $\tau_{k}$, since it has positive
	leading coefficient and negative constant term,
	it has a positive and a negative real root. 
	Thus, if \eqref{eq:nr2}
	holds for some value of $\tau_{k}>1>0$, then also for any larger value.
	Hence in view of Lemma~\ref{nazur} it suffices to have \eqref{eq:nr2}
	for $\tau_{k}=\chi_{k}=(\nu_{k}-n+1)/n$.
	Inserting and expanding, we derive the sufficient hypothesis
	\eqref{eq:sech}.
	Similarly, we combine \eqref{eq:nr3}
	with \eqref{eq:jup}
	to obtain the criterion \eqref{eq:sech2}, hereby using $\nu_{k+1}-\tau_{k+1}>0$ as a consequence of \eqref{eq:este7} and $\widehat{w}_{n}(\theta)\geq n\geq 2$, $\nu_{k+1}>1$
	by \eqref{eq:diri}.
	%
	%Then we apply the method
	%from the proof of Theorem~\ref{cc} with $\ell=3$ to derive that if
	%
	%\[
	%2< \frac{\log\left( %\frac{(w-2n+1)(w_{n}(theta)-n)}{n}+1\right)}{\log(w_{n}(theta)/n)}
	%\]
	%
	%then $P_{k},P_{k+1},P_{k+2}$ are linearly independent (where we take
	%the limit accordingly if $w_{n}(theta)=\infty$). We see that
	%for $w=\chi$ large enough the inequality holds.
\end{proof}

For the proof of Theorem~\ref{konse} we essentially proceed as in 
the proof of Theorem~\ref{rr}.

\begin{proof}[Proof of Theorem~\ref{konse}]
	Assume the opposite that $\mathscr{B}_{k}$ is linearly dependent.
	Then we have a polynomial identity
	\[
	P_{k}U_{k}+\cdots +P_{k+\ell-1}U_{k+\ell-1} \equiv 0,
	\]
	with $U_{k}$ integer polynomials of degree at most $d$, not all
	identically $0$. The identity can be written
	in coordinates in form of a linear equation system $B\underline{a}=\underline{0}$
	with $B$ a matrix with $n+d+1$ rows 
	and $(d+1)\ell$ columns whose
	entries are coefficients of the polynomials $P_{j}$, 
	and $\underline{a}\in\mathbb{Z}^{(d+1)\ell}$ the vector
	consisting of the coefficients of all $U_{j}$. 
	By considering if necessary a maximum linearly independent subset of the 
	columns, and distinguishing
	the cases where some of the remaining columns originates
	from $P_kU_k$ and where this is not the case, we
	can assume that $B$ above has corank $1$ and the first polynomial
	$U_{k}$ does not vanish in any such non-trivial solution.
	This indeed works very similar to the proof of Theorem~\ref{siegele},
	we leave the details to the reader,
	and means $w=(d+1)\ell-1$ in sense of notation in the proof of Theorem~\ref{siegele}.
	Then again we obtain a one-dimensional solution lattice for $\underline{a}$.
	
	Since there is a non-trivial solution the matrix $B$ 
	has rank less than $(d+1)\ell$
	and application of Siegel's Lemma
	when trivially estimating the subdeterminants by $X^{(d+1)\ell-1}$ 
	now gives that each $U_{j}$ has height
	$H(U_{j})\ll_{n} X^{(d+1)\ell-1}$, where
	again $X:=H(P_{k+\ell-1})$.
	Then since $U_{k}$ does not vanish identically we can estimate
	\[
	\vert U_{k}(\theta)P_{k}(\theta)\vert = \vert U_{k}(\theta)\vert \cdot 
	\vert P_{k}(\theta)\vert\gg H(U_{k})^{-w_{d}(\theta)-\epsilon}\cdot 
	\vert P_{k}(\theta)\vert\gg X^{-[(d+1)\ell-1]\cdot(w_{d}(\theta)+\epsilon)}\cdot 
	\vert P_{k}(\theta)\vert. 
	\]
	In view of 
	\[
	\vert P_{k}(\theta)U_{k}(\theta)\vert= \vert P_{k+1}(\theta)U_{k+1}(\theta)+ \cdots +P_{k+\ell-1}(\theta)U_{k+\ell-1}(\theta)\vert\ll_{n} \vert P_{k+1}(\theta)\vert \max H(U_{j}),
	\]
	similar to Theorem~\ref{rr} we obtain the relation
	\[
	X^{(d+1)\ell-1}\vert P_{k}(\theta)\vert^{\underline{\sigma}-\epsilon}\gg_{n} \vert P_{k}(\theta)\vert \cdot X^{-[(d+1)\ell-1]\cdot(w_{d}(\theta)+\epsilon)}.
	\]
	By $X\ll H(P_{k})^{\overline{\tau}^{\ell-1}+\epsilon}$ we conclude
	\[
	\vert P_{k}(\theta)\vert \geq H(P_{k})^{-\mu+\varepsilon}, \qquad 
	\mu=\frac{[(d+1)\ell-1](w_{d}(\theta)+1)\overline{\tau}^{\ell-1}}{\underline{\sigma}-1}>0,
	\]
	for $\varepsilon>0$ some modification of $\epsilon$.
	%As this holds for arbitrarily large $k$ we infer
	%
	%\[ 
	%w_{n}(theta)\leq \frac{\ell \tau^{\ell-1}}{\sigma-1}.
	%	\]
	On the other hand since $P_{k}$ is a best approximation polynomial, 
	by Dirichlet's Theorem \eqref{eq:didit} we have
	$\vert P_{k}(\theta)\vert \ll_{n} H(P_{k})^{- \underline{\tau}\widehat{w}_{n}(\theta)+\epsilon}$.
	Combining yields
	\[
	n\leq  \widehat{w}_{n}(\theta)\leq \frac{[(d+1)\ell-1)](w_{d}(\theta)+1)\overline{\tau}^{\ell-1}}
	{(\underline{\sigma}-1)\underline{\tau}}+\varepsilon.
	\]
	Hence again assuming the reverse inequality and letting $\epsilon\to 0$
	and thus $\varepsilon\to 0$
	we cannot have the assumed linear
	dependence relation.
\end{proof}

\section{Annex: A proof of Theorem~\ref{mothm}}  \label{annex}

The following proof of the claim \eqref{eq:neujaa} stated in
the unpublished online resoruce~\cite{moneu} was pointed out to the author by the referee.

First consider the generic case $h=m+n$ only. 
Denote $\underline{z}_{k}= (x_{1,k},\ldots,x_{n,k},y_{1,k},\ldots,y_{m,k})$
for $k\geq 1$ the $k$-th minimal point associated to $\Theta$.
If $m+n$ consecutive minimal points $\underline{z}_{t},\ldots,\underline{z}_{t+m+n-1}$
are linearly independent, their determinant 
\[
\Delta=\begin{vmatrix}
x_{1,t} & \cdots & x_{n,t} & y_{1,t} & \cdots & y_{m,t} \\ 
x_{1,t+1} & \cdots & x_{n,t+1} & y_{1,t+1} & \cdots & y_{m,t+1} \\ 
\cdots & \cdots & \cdots & \cdots & \cdots & \cdots \\
x_{1,t+m+n-1} & \cdots & x_{n,t+m+n-1} & y_{1,t+m+n-1} & \cdots & y_{m,t+m+n-1}
\end{vmatrix}
\]
is non-zero. By taking linear combinations of columns 
we can make any $y_{j,k}$
smaller than $\Vert\Theta^E \underline{z}_{k}\Vert$, for
$1\leq k\leq m$ and $t\leq j\leq t+m+n-1$, and we estimate
\begin{align*}
1\leq |\Delta| \ll \Vert \underline{z}_{t+m+n-1}  \Vert \cdot
\Vert \underline{z}_{t+m+n-2}  \Vert \cdots \cdot \Vert \underline{z}_{t+n}  \Vert
\cdot \Vert \Theta^{E} \underline{z}_{t+n-1}  \Vert \cdot \Vert \Theta^{E} \underline{z}_{t+n-2}  \Vert \cdots \cdot \Vert \Theta^{E} \underline{z}_{t}  \Vert.
\end{align*}
By definition of $\widehat{w}(\Theta)$ for $\varepsilon>0$
and $t\geq t_{0}(\varepsilon)$ it gives
\[
1\ll \Vert \underline{z}_{t+m+n-1}  \Vert \cdot
\Vert \underline{z}_{t+m+n-2}  \Vert \cdots \cdot \Vert \underline{z}_{t+n+1}  \Vert\cdot \Vert \underline{z}_{t+n}  \Vert^{1-\widehat{w}(\Theta)+\varepsilon}\cdot \Vert \underline{z}_{t+n-1}  \Vert^{-\widehat{w}(\Theta)+\varepsilon} \cdots \cdot \Vert \underline{z}_{t+1}  \Vert^{-\widehat{w}(\Theta)+\varepsilon}.
\]
One checks that this implies for some $j\in \{ t+1,\ldots,t+n-2 \}$ the estimate
\[
\Vert \underline{z}_{j+1}  \Vert \geq \Vert \underline{z}_{j}  \Vert^{G_{m,n}-\epsilon }
\]
with $G_{m,n}$ as in the theorem and
$\epsilon>0$ that tends to $0$ as $\varepsilon$ does. 
Thus $\tau_{j}\geq G_{m,n}-\epsilon$ in our notation
from Section~\ref{cri}. From $\textbf{L}_{m+n}$ it follows
there are arbitrarily large such $t$, and since $\epsilon>0$ can be arbitrarily small Lemma~\ref{lemur} implies $w(\Theta)/\widehat{w}(\Theta) \geq \overline{\tau} \geq G_{m,n}$.

For general $h$, one can reduce the problem to $m^{\prime}\times n^{\prime}$
matrices by considering the subspace $\mathscr{L}_{\Theta}\cap \mathscr{S}_{\Theta}$ with
$\mathscr{L}_{\Theta}$ and $\mathscr{S}_{\Theta}$
as in the theorem. 

\vspace{1cm}

{\em The author thanks the referee for many helpful remarks and references that led to many improvements of the results and a clearer exposition
	of the original version.}

\end{document}